\documentclass[12pt]{article}
\usepackage[margin=0.7in]{geometry}
\usepackage{amssymb,amsmath,amsthm,mathrsfs}
\allowdisplaybreaks
\usepackage[alphabetic]{amsrefs}
\usepackage [latin1]{inputenc}
\usepackage{color}
\newcommand{\R}{\mathbb{R}}

\newtheorem{thm}{Theorem}[section]
\newtheorem{cor}[thm]{Corollary}
\newtheorem{lem}[thm]{Lemma}
\newtheorem{prop}[thm]{Proposition}

\newtheorem{rem}[thm]{Remark}

\begin{document}

\title{Pointwise gradient bounds for entire solutions \\ of elliptic equations with non-standard growth conditions \\ and general nonlinearities}

\author{Cecilia Cavaterra${}^{(1,5)}$\and
Serena Dipierro${}^{(2)}$\and
Alberto Farina${}^{(3)}$
\and
Zu Gao${}^{(4)}$
\and
Enrico Valdinoci${}^{(2,1)}$
}

\maketitle

{\scriptsize \begin{center} (1) -- Dipartimento di Matematica ``Federigo Enriques''\\
Universit\`a degli studi di Milano\\
Via Saldini 50, I-20133 Milano (Italy)\\
\end{center}
\scriptsize \begin{center} (2) -- Department of Mathematics and Statistics\\
University of Western Australia\\ 35 Stirling Highway, WA6009 Crawley (Australia)\\
\end{center}
\scriptsize \begin{center} (3) --
LAMFA,
CNRS UMR 7352\\
Facult\'e des Sciences\\
Universit\'e de Picardie Jules Verne\\
33 rue Saint Leu,
80039 Amiens CEDEX 1 (France)\\ \end{center}
\scriptsize \begin{center} (4) --
School of Mathematics and Statistics\\
Central South University\\
932 Lushan S Road\\
410083 Hunan, Changsha (China) \end{center}
\scriptsize \begin{center} (5) -- Istituto di Matematica Applicata e Tecnologie Informatiche ``Enrico Magenes'', CNR\\
   Via Ferrata 1, 27100 Pavia (Italy)\\
\end{center}

\bigskip

\begin{center}
E-mail addresses:
{\tt cecilia.cavaterra@unimi.it},
{\tt serena.dipierro@uwa.edu.au},
{\tt alberto.farina@u-picardie.fr},
{\tt gaozu7@163.com},
{\tt enrico.valdinoci@uwa.edu.au}
\end{center}
}
\bigskip\bigskip
\begin{abstract}
We give pointwise gradient bounds for solutions
of (possibly non-uniformly)
elliptic partial differential equations in the entire Euclidean space.

The operator taken into account is very general
and comprises also the singular and degenerate
nonlinear case with
non-standard growth conditions.
The sourcing term is also allowed to have a very general form,
depending on the space variables,
on the solution itself, on its gradient, and possibly on higher order derivatives
if additional structural conditions are satisfied.
\end{abstract}
\smallskip

\smallskip

\medskip

\section{Introduction}
\vskip2mm
\noindent
In this paper we consider a very general elliptic equation,
set in the whole of the Euclidean space, and we will establish
pointwise gradient bounds for the solutions.
The operator taken into account can be degenerate and singular,
and we can also consider the case of the superposition of differential operators
with different homogeneity.

The main result establishes that (a possibly nonlinear function of)
the gradient of the solution is bounded at any point by a suitable potential
function. Moreover, the bound obtained, which can be
seen as a generalization of the Energy Conservation Principle
to PDEs, is in general sharp, since if equality is attained
in this bound, the solution is shown to be necessarily constant.

Our results comprise, as particular cases,
the classical results in~\cites{MR803255, MR1296785}.
The method of proof is based on Maximum Principles
and it can be seen
as a refinement of the classical Bernstein method introduced
in~\cite{MR1544873}, as extended in~\cite{MR0454338, MR583337, MR615561}.
Namely, one considers a suitable auxiliary function, called ``$P$-function''
in jargon, which is
defined in terms of the solution and its gradient, and shows
that such a $P$-function satisfies a differential inequality: from this
and the Maximum Principle, the desired bounds on the gradient plainly
follow.\medskip

In spite of its intrinsic simplicity
(and unquestionable beauty), the idea of obtaining gradient bounds via the Maximum Principle
turned out to be very effective, and it found several applications
in many topics, including Riemannian geometry (see e.g.~\cites{MR0385749, MR1230276, MR2285258, MR2812957})
anisotropic or nonhomogeneous equations (see e.g.~\cites{MR1942128, MR3158523, MR3168616, MR3231999, MR3348935, MR3587074}),
and also subelliptic equations (see~\cite{MR2545524}),
and, when the equation is set in a domain, the technique also
detects the geometry of the domain itself (see e.g.~\cite{MR2680184, MR2911121}).
Moreover, a novel approach to the Maximum
Principle method has been recently exploited
in a very successful way in~\cites{MR3125548, MR3381494, 2018arXiv180809615A},
in order to obtain oscillation
and modulus of continuity estimates.
In general, these types of gradient and continuity estimates
are also related to rigidity results for overdetermined problems (see e.g.~\cites{MR980297, MR2591980, MR3145008})
and they
also provide, as a byproduct, new classification results
of Liouville type (see also~\cites{MR1674355, MR2317549}).\medskip

In the case under consideration in this paper, given the very general structure
of the equation, one needs to exploit a technique introduced
in~\cite{MR3049726}: in our case, such differential inequality
will be satisfied, in general, only up to a reminder, which can be shown
to have the appropriate sign in a number of concrete examples.\medskip

Let us now describe in detail the mathematical framework
in which we work.
We consider the following PDE in divergence form:
\begin{equation}\label{a1.4}
\mathrm{div}(\Phi'(|\nabla u|^2)\nabla u)=f(u)+g(\nabla u, Su)\qquad
{\mbox{ in }}\;\R^n,
\end{equation}
where $\Phi\in C^{3,\alpha}_{\mathrm{loc}}\big((0,+\infty)\big)
\cap C\big([0,+\infty)\big)$ for some $\alpha\in(0,1)$, with~$\Phi(0)=0$,
$f\in C^1(\mathbb{R})$ and~$g\in C^{1}(\mathbb{R}^n\times
\mathbb{R}^{N-n})\cap L^{\infty}(\mathbb{R}^n\times\mathbb{R}^{N-n})$.

We denote by~$(\zeta, \eta)\in \mathbb{R}^n\times\mathbb{R}^{N-n}$
the variables of the function~$g$, i.e.,~$g:=g(\zeta, \eta)$, and
we assume that
for all~$M>0$ we have that
\begin{equation} \label{SOMA}
\sup_{{(\zeta, \eta)\in \mathbb{R}^n\times\mathbb{R}^{N-n}}\atop{|\zeta|\le M}} \big|
g_{\zeta_j}(\zeta,\eta)
\big|<+\infty,\quad{\mbox{ for all }}j\in\{1,\dots,n\},
\end{equation}
where the subscript~$\zeta_j$ denotes
partial derivative with respect to the variable~$\zeta_j$.

In~\eqref{a1.4} and throughout this article,
$$S: L^\infty(\mathbb{R}^n)\cap C^{\ell}(\mathbb{R}^{n})\mapsto
\big(C^{\ell'}(\mathbb{R}^{n})\big)^{N-n}$$
will denote an operator\footnote{For instance, in our setting,
$$ S(u)=\Big(x,\;u,\;x+\nabla u,\; u_{111}-u_{22}+\Delta^2 u, \;
u_{11}^5, \; x\cdot\nabla u+\log(1+u_{2222}^4)+\sqrt{-\Delta}(\arctan u)\Big)
$$
is an admissible (though not specially meaningful)
operator. In this case, $N-n=n+1+n+3$, that is~$N=3n+4$.
In our
setting, it is an interesting
feature that the nonlinear source~$g$ can also depend on higher derivatives,
on nonlinear differential operators, on integro-differential operators, etc.}
acting on bounded and smooth functions, with~$\ell\in[3,+\infty]$ and~$\ell'\in[1,+\infty]$,
and we will write
$S=(S^{[1]},\dots, S^{[N-n]})$ where $N\geqslant n\geqslant1$. If $N-n=0$, we have that $g$ does not depend on the
variable~$\eta$.

We stress that~$S$ is just a map sending functions into vectorial functions,
and it does not necessarily need to be linear or continuous in any topology.
Also, for the sake of simplicity,
we will consider smooth\footnote{In this paper, we did not
optimize the
regularity assumptions\label{REGO67}
on the solution~$u$.
For our purposes, it is sufficient
to have sufficient regularity to write~\eqref{a1.4}
in the pointwise sense and consider its derivatives.
Hence, if the operator~$S$ only involves a finite number of derivatives,
then also~$u$ is required to have a finite number of derivatives.
When~$S$ only involves operators of order~$1$ or less, in concrete cases
one can also apply standard elliptic regularity theory to obtain the desired regularity
of~$u$ starting with rather minimal assumptions. Since the minimal regularity
assumptions in this general setting are rather technical, we will not
introduce this additional complication in this article, sticking to the case
of sufficiently
smooth solutions.}
solutions~$u\in C^{\ell}(\mathbb{R}^n)$
of~\eqref{a1.4}.
\medskip

As customary, we will assume that the divergence form operator in~\eqref{a1.4}
possesses suitable (possibly singular or degenerate) elliptic structure,
which will ensure the validity of the Maximum Principle.
For this,
for any~$\sigma\in\mathbb{R}^{n}$, we set
\begin{equation}\label{a1.5}
a_{ij}(\sigma):=2\Phi''(|\sigma|^2)\sigma_i\sigma_j+\Phi'(|\sigma|^2)\delta_{ij},
\end{equation}
and
we will always assume in this paper
that at least one of the following Assumptions~A
and~B is satisfied:\medskip

\noindent{\bf Assumption A.}
There exist~$p>1$, $a\ge0$ and~$C_1$, $C_2>0$ such that, for every~$
\sigma$, $\xi\in
\R^n\setminus\{0\}$,
\begin{equation}\label{083-384-1}
C_1(a+|\sigma|)^{p-2}\le\Phi'(|\sigma|^2)\le C_2(a+|\sigma|)^{p-2}\\
\end{equation}
\begin{equation}\label{a1.7}
{\mbox{and }}\qquad\qquad
C_1(a+|\sigma|)^{p-2}|\xi|^2\le\sum_{i,j=1}^n a_{ij}(\sigma)\xi_i\xi_j
\le C_2(a+|\sigma|)^{p-2}|\xi|^2.
\end{equation}
\medskip

\noindent{\bf Assumption B.}
We have that~$\Phi\in C^1([0,+\infty))$, and
there exist~$C_1$, $C_2>0$ such that, for every~$
\sigma\in\R^n$ and every~$\xi'=(\xi,\xi_{n+1})\in
\R^n\times\R$, with~$\xi'\cdot(-\sigma,1)=0$,
\begin{equation}\label{083-384-2}
C_1(1+|\sigma|)^{-1}\le\Phi'(|\sigma|^2)\le C_2(1+|\sigma|)^{-1}
\end{equation}
\begin{equation}\label{a1.9}
{\mbox{and }}\qquad\qquad
C_1(1+|\sigma|)^{-1}|\xi'|^2\le\sum_{i,j=1}^n a_{ij}(\sigma)\xi_i\xi_j
\le C_2(1+|\sigma|)^{-1}|\xi'|^2.
\end{equation}
\smallskip

Related structural assumptions on the diffusive operators have been
considered in~\cite{MR1296785, MR3049726}.
We observe that Assumptions~A and~B
will be enforced with~$\sigma:=\nabla u$,
hence, under a Lipschitz condition on the solution~$u$,
one has that~$|\nabla u|\le M$ for some~$M>0$.
So that it will be sufficient to require Assumptions~A and~B
with~$\sigma$ belonging to the ball of radius~$M$ centered
at the origin, which we will denote by~$B_M$.
Therefore, from now on, when we say that Assumptions~A and~B
are satisfied, we mean that they are fulfilled when~$\sigma\in B_M$,
and the constants~$C_1$ and~$C_2$ can depend on~$M$.
In particular, when Assumption~B is in force, we can reduce
to Assumption~A with~$p=2$, with constants depending on~$M$.

In our setting,
we have that Assumptions~A and~B are satisfied
by very general nonlinear operators, as established by the following
result:

\begin{prop}\label{STRUCT}
Let~$m\ge1$ and
\begin{equation}\label{PHI}
\Phi(r):=\sum_{k=1}^m \left( \frac{2c_k}{p_k}(b_k+r)^{\frac{p_k}{2}}-\frac{2c_kb_k^{\frac{p_k}2}}{p_k}\right),
\end{equation}
with
\begin{equation}\label{pj}
1\leq p_1\leq\ldots\leq p_m<+\infty,
\end{equation}
and
\begin{equation}\label{cj}
c_k>0,\mathrm{~for~every~}k\in\{1,\dots,m\}.
\end{equation}
Then:
\begin{itemize}
\item[{(i)}] If
\begin{equation}\label{EITH1}
\begin{split}&
p_1>1,\qquad\qquad
b_1\ge0\\
{\mbox{and }}\quad&
\mu b_1\le b_k\le\frac{b_1}\mu,\quad{\mbox{ for all }}k\in\{1,\dots,m\},\end{split}\end{equation}
for some~$\mu\in(0,1)$, then Assumption~A holds true.
\item[{(ii)}] If
\begin{equation}\label{EITH2} \mu \le b_k\le\frac{1}\mu,\quad{\mbox{ for all }}k\in\{1,\dots,m\},
\end{equation}
for some~$\mu\in(0,1)$, then Assumption~B holds true.
\end{itemize}
\end{prop}

In view of Proposition~\ref{STRUCT} it follows that
Assumptions~A and~B comprise the important case
of {\em nonlinear operators with
non-standard growth conditions and with non-uniform ellipticity properties},
see~\cites{MR1814973, MR3614673, MR3775180}.
\medskip

In our setting,
the bounds on the gradient of the solution~$u$ will require the control on the sign of
a suitable reminder. To describe this feature in details, we give some notation.
For any~$r\in\R$, we define
\begin{equation}\label{a2.1}
\Lambda(r):=2r\Phi''(r)+\Phi'(r).
\end{equation}
In this setting, the reminder function that we consider is defined\footnote{We take
this opportunity to amend a flaw in~\cite{MR3049726}.
As a matter of fact, due to a cut-and-paste error, the term
$$ \frac{2f(u)|\nabla u|^2}{\Lambda(|\nabla u|^2)}\sum^{n}_{j=1}g_{p_j}(x,u,\nabla u)u_j$$
is missing from formula~(1.13) in~\cite{MR3049726}. The proof
in~\cite{MR3049726} (which is based on Lemma~2.1 there)
is however {\em correct} as it is.
Formula~(1.11) and Remark~1.4 of~\cite{MR3049726} have also
to be corrected by adding the missing term (e.g., saying that~$fg_{p_i}p_i \ge0$). Also,
for clarity,
we point out some minor typos in~\cite{MR3049726}:
the statement~``$w\in{\mathcal{S}}$''
three lines below~(3.4) should be~``$w\in\overline{\mathcal{S}}$'',
the set~$V$ on line~3 of page~625
should be corrected into~$\mathscr{V}$, the ``neighborhood of~$x_o$''
in the last line of the proof of Theorem~1.3
should be the ``neighborhood of~$y$''.
Also, throughout all~\cite{MR3049726}, the function~$ g$
is assumed to be uniformly in~$C^{0,\alpha}$ with respect to the $x$ variable.}
on~$\{\nabla u\ne0\}$ by
\begin{equation}\label{a1.13}
\aligned
\mathscr{R}:=&-\frac{2f(u)g(\nabla u, Su)|\nabla u|^2}{\Phi'(|\nabla u|^2)}+2|\nabla u|^2\sum^{n}_{k=1}\sum^{N-n}_{j=1}g_{\eta_j}(\nabla u, Su)\frac{\partial S^{[j]}u}{\partial x_k}u_{k}\\
~~~~~~~~~~~~~~~~&+\frac{2f(u)|\nabla u|^2}{\Lambda(|\nabla u|^2)}\sum^{n}_{j=1}g_{\zeta_j}(\nabla u,Su)u_j.
\endaligned\end{equation}
As customary, if~$N-n=0$ the second term in the right hand side of~\eqref{a1.13}
is considered to be zero (equivalently, in this case, the function~$g$
does not depend on the variable~$\eta$).

Given~$\alpha\in(0,1]$
we will also denote by~$C^{0,\alpha}(\R^n)$ the space of functions~$u\in L^\infty(\R^n)$
such that
$$ \sup_{{x,y\in\R^n}\atop{x\ne y}}\frac{ |u(x)-u(y)|}{|x-y|^\alpha}<+\infty.$$
\medskip

In this framework, our pivotal result is the following:

\begin{thm}\label{T1.2} Assume that~$u\in C^\ell(\R^n)\cap W^{1,\infty}(\R^n)$
is a solution of~\eqref{a1.4}. For every~$r\in\R$, let
\begin{equation}\label{POTE} F_0(r):=\int_0^r f(\tau)\,d\tau,\qquad c_u:=\inf_{x\in\R^n} F_0(u(x))\qquad{\mbox{and}}\qquad
F(r):=F_0(r)-c_u.\end{equation}
Assume that
\begin{equation}\label{CONDalpha}
Su\in C^{0,\alpha}(\mathbb{R}^n,\mathbb{R}^{N-n}),\qquad{\mbox{ for some }}
\alpha\in(0,1]\end{equation}
and
\begin{equation}\label{COND}
\mathscr{R}(x)\ge0,\qquad{\mbox{for every $x\in\{ \nabla u\ne0\}$}}.
\end{equation}
Then,
\begin{equation}\label{GRAD}
2\Phi'\big(|\nabla u(x)|^2\big)|\nabla u(x)|^2-\Phi\big(|\nabla u(x)|^2\big)\le2F(u(x)),
\qquad{\mbox{for every $x\in\R^n$.}}
\end{equation}
\end{thm}

We observe that, since~$u$ is bounded, we have that~$c_u$ is finite
and the setting in~\eqref{POTE}
is well posed. As a matter of fact, such a setting can be seen as a ``gauge''
on the potential function that makes~$F$ nonnegative on the range
of the solution.

Condition~\eqref{CONDalpha} can be seen as a regularity
assumption on the solution (it can be also relaxed, for instance,
if~$Su(x)=(x,Tu(x))$, with~$T: L^\infty(\mathbb{R}^n)\cap C^{\ell}(\mathbb{R}^{n})\mapsto
\big(C^{\ell'}(\mathbb{R}^{n})\big)^{N-2n}$, it is enough to
suppose that~$g$ is uniformly~$C^{0,\alpha}$ in the $x$ variable
and~$Tu\in C^{0,\alpha}(\mathbb{R}^n,\mathbb{R}^{N-2n})$).

We also point out that Theorem~\ref{T1.2} comprises,
as special cases, some classical results. In particular,
when~$\Phi(r):=r$ and~$g$ vanishes identically,
then~$\mathscr{R}$ also vanishes identically, hence condition~\eqref{COND}
is satisfied. In this case, equation~\eqref{a1.4} reduces to
$$ \Delta u=f(u),$$
and~\eqref{GRAD} boils down to
$$ |\nabla u(x)|^2\le2F(u(x)),$$
which is precisely the classical result in~\cite{MR803255}.
Similarly, some results in~\cite{MR1296785} and~\cite{MR3049726}
are also recovered
as particular cases of Theorem~\ref{T1.2} (and, from the technical
point of view, the setting introduced here simplifies
and extends that in~\cite{MR3049726}, by keeping track
at the same time of all the derivatives of the nonlinear source~$g$).
In particular, recovering the elliptic regularity theory
as mentioned in the footnote on page~\pageref{REGO67},
one can obtain from Theorem~\ref{T1.2} the classical
results in~\cite{MR803255}, \cite{MR1296785} and~\cite{MR3049726}
also for weak solutions.\medskip

In some sense, one can consider Theorem~\ref{T1.2}
as an {\em abstract} result, in which a very general framework is taken
into account, with minimal structural assumptions on the equation,
but under a fundamental condition on the sign of the reminder
function, as given in~\eqref{COND}. To apply this result to particular
cases of interest,
we point out now that condition~\eqref{COND} is indeed
satisfied in a number of {\em concrete} situations,
such as the $p$-Laplacian operator, the graphical mean curvature
operators, and operators obtained by the superposition
of singular and degenerate operators with different scaling properties,
proving gradient bounds under simple structural assumptions on the nonlinear sources.
Indeed, we have the following result:

\begin{prop}\label{CASE}
Let~$m\ge1$ and $\Phi$ be as in~\eqref{PHI},
under assumptions~\eqref{pj}
and~\eqref{cj}, and suppose that~$b_k\ge0$ for all~$k\in\{1,\dots,m\}$.
Assume that
\begin{equation}\label{Su} S(u):=u.\end{equation}
Suppose also that
\begin{eqnarray}
\label{LAMONO}&&{\mbox{$g(\zeta,\eta)\le g(\zeta,\tilde\eta)$ for all~$\zeta\in\R^n$ and~$\eta\le \tilde\eta$,}}\\
\label{LAMONO2}&&{\mbox{and $g(\lambda\zeta,\eta)=\lambda^\beta g(\zeta,\eta)$
for all~$\lambda>0$, for some~$\beta>0$}}
\end{eqnarray}
In addition, assume that one of the following five conditions is satisfied:
either
\begin{equation}\label{YAU1}
{\mbox{$m=1$, $\beta=p_1-1$
and $(p_1-2)f(\eta)g(\zeta,\eta)\ge0$,
for all~$\zeta\in\R^n$ and~$\eta\in\R$}},\end{equation}
or
\begin{equation}\label{YAU4}
{\mbox{$m=1$, $p_1=2$
and $(\beta-1)f(\eta)g(\zeta,\eta)\ge0$,
for all~$\zeta\in\R^n$ and~$\eta\in\R$}},\end{equation}
or
\begin{equation}\label{YAU5}
{\mbox{$m=1$, $\beta=1$
and $(2-p_1)f(\eta)g(\zeta,\eta)\ge0$,
for all~$\zeta\in\R^n$ and~$\eta\in\R$}},\end{equation}
or
\begin{equation}\label{YAU2}
{\mbox{$\beta\geq\max\{1, p_m-1\}$ and $f(\eta)g(\zeta,\eta)\ge0$,
for all~$\zeta\in\R^n$ and~$\eta\in\R$,}}\end{equation}
or
\begin{equation}\label{YAU3}
{\mbox{$b_1=\dots=b_m=0$,
$\beta\leq p_1-1$
and $f(\eta)g(\zeta,\eta)\le0$,
for all~$\zeta\in\R^n$ and~$\eta\in\R$,}}
\end{equation}
or
\begin{equation}\label{YAUF}
{\mbox{$m=1$, $b_1=0$ and
$\beta=p_1-1$.}}
\end{equation}
Then~$\mathscr{R}\ge0$.
\end{prop}

A concrete example that satisfies assumptions~\eqref{LAMONO}
and~\eqref{LAMONO2} is
$$ g(\zeta,\eta)=|\zeta|^\beta \, h(\eta),$$
with~$\beta>1$ and~$h$ increasing.

Other concrete situations in which one can explicitly check that~$\mathscr{R}\ge0$
will be discussed in the forthcoming Remarks~\ref{ESEMPI}
and~\ref{ESEMPI2}.\medskip

Combining Theorem~\ref{T1.2} with Propositions~\ref{STRUCT}
and~\ref{CASE}, we plainly
obtain the following gradient estimate in a very general,
but concrete, situation:

\begin{cor}\label{CASE2}
Let~$m\ge1$ and
\begin{equation*}
\Phi(r):=\sum_{k=1}^m \left( \frac{2c_k}{p_k}(b_k+r)^{\frac{p_k}{2}}-\frac{2c_kb_k^{\frac{p_k}2}}{p_k}\right),
\end{equation*}
with~$1\leq p_1\leq\ldots\leq p_m<+\infty$
and~$c_k>0$ for every~$k\in\{1,\dots,m\}$.

Suppose that either
\begin{equation}\label{EITH1X}
\begin{split}&
p_1>1,\qquad\qquad
b_1\ge0\\
{\mbox{and }}\quad&
\mu b_1\le b_k\le\frac{b_1}\mu,\quad{\mbox{ for all }}k\in\{1,\dots,m\},\end{split}\end{equation}
or
\begin{equation} \label{EITH2X}\mu \le b_k\le\frac{1}\mu,\quad{\mbox{ for all }}k\in\{1,\dots,m\},
\end{equation}
for some~$\mu\in(0,1)$.

Suppose also that~$g:\R^n\times\R\to\R$ satisfies the following
monotonicity and homogeneity assumptions:
\begin{eqnarray}
\label{LAMONOX}&&{\mbox{$g(\zeta,\eta)\le g(\zeta,\tilde\eta)$ for all~$\zeta\in\R^n$ and~$\eta\le \tilde\eta$,}}\\
\label{LAMONO2X}&&{\mbox{and $g(\lambda\zeta,\eta)=\lambda^\beta g(\zeta,\eta)$
for all~$\lambda$, $\beta>0$.}}
\end{eqnarray}
In addition, assume that one of the following five conditions is satisfied:
either
\begin{equation}\label{YAU1X}
{\mbox{$m=1$, $\beta=p_1-1$
and $(p_1-2)f(\eta)g(\zeta,\eta)\ge0$,
for all~$\zeta\in\R^n$ and~$\eta\in\R$}},\end{equation}
or
\begin{equation}\label{YAU4X}
{\mbox{$m=1$, $p_1=2$
and $(\beta-1)f(\eta)g(\zeta,\eta)\ge0$,
for all~$\zeta\in\R^n$ and~$\eta\in\R$}},\end{equation}
or
\begin{equation}\label{YAU5X}
{\mbox{$m=1$, $\beta=1$
and $(2-p_1)f(\eta)g(\zeta,\eta)\ge0$,
for all~$\zeta\in\R^n$ and~$\eta\in\R$}},\end{equation}
\begin{equation}\label{YAU2X}
{\mbox{$\beta\geq\max\{1, p_m-1\}$ and $f(\eta)g(\zeta,\eta)\ge0$,
for all~$\zeta\in\R^n$ and~$\eta\in\R$,}}\end{equation}
or
\begin{equation}\label{YAU3X}
{\mbox{$b_1=\dots=b_m=0$,
$\beta\leq p_1-1$
and $f(\eta)g(\zeta,\eta)\le0$,
for all~$\zeta\in\R^n$ and~$\eta\in\R$,}}
\end{equation}
or
\begin{equation}\label{YAUFX}
{\mbox{$m=1$, $b_1=0$ and
$\beta=p_1-1$.}}
\end{equation}
Assume that~$u\in C^\ell(\R^n)\cap W^{1,\infty}(\R^n)$
is a solution of
\begin{equation}\label{DIVS} \mathrm{div}(\Phi'(|\nabla u|^2)\nabla u)=f(u)+g(\nabla u,u),
\qquad{\mbox{in }}\;\R^n.\end{equation}
For every~$r\in\R$, let
\begin{equation*}
F_0(r):=\int_0^r f(\tau)\,d\tau,\qquad c_u:=\inf_{x\in \R^n} F_0(u(x))\qquad{\mbox{and}}\qquad
F(r):=F_0(r)-c_u.\end{equation*}
Then,
\begin{equation}\label{GRADX}
2\Phi'\big(|\nabla u(x)|^2\big)|\nabla u(x)|^2-\Phi\big(|\nabla u(x)|^2\big)\le2F(u(x)),
\qquad{\mbox{for every $x\in\R^n$.}}
\end{equation}
\end{cor}

{\begin{rem}\label{ESEMPI}\rm Checking condition~\eqref{COND} can be, in principle,
not a trivial task in practice. Nevertheless, there are a number
of concrete cases in which condition~\eqref{COND} is automatically satisfied.
Without any attempt of being exhaustive, and only for the sake
of confirming the interest of such a condition, we list here some
of these situations in which condition~\eqref{COND} is fulfilled.
For simplicity, we focus here on the case in which
$f$ vanishes identically, and thus~\eqref{a1.13}
reduces to
\begin{equation}\label{RED}
\frac{\mathscr{R}}{2|\nabla u|^2}=
\sum^{n}_{k=1}\sum^{N-n}_{j=1}g_{\eta_j}(\nabla u, Su)
\frac{\partial S^{[j]}u}{\partial x_k}u_{k}.\end{equation}

\noindent{\bf (i).} An interesting example is given by the equation
\begin{equation}\label{a1.4X91023}\mathrm{div}(\Phi'(|\nabla u|^2)\nabla u)=h(u)+c(x)\cdot\nabla u\qquad
{\mbox{ in }}\;\R^n,\end{equation}
with~$h\in C^1(\R)$, $h'\ge0$, $c=(c_1,\dots,c_n)\in C^1(\R^n,\R^n)\cap L^\infty(
\R^n,\R^n)$, under the assumption
that the matrix~$\left\{ \frac{\partial c_j}{\partial x_k}\right\}_{j,k\in\{1,\dots,n\}}$
is nonnegative definite.

To check that~\eqref{COND} is satisfied in this case, it is convenient
to take~$N:=2n+1$, $Su:=(c,h(u))$, that is
$$ S^{[j]}u:=\begin{cases}
c_j & {\mbox{ if }}j\in\{1,\dots,n\},\\
h(u) & {\mbox{ if }}j=n+1,
\end{cases}$$
and
$$ g(\zeta,\eta)=g(\zeta_1,\dots,\zeta_n,\eta_1,\dots,\eta_{n+1}):=
\sum_{j=1}^n \zeta_j\eta_j +\eta_{n+1}.$$
Notice that, with this choice, the general setting in~\eqref{a1.4}
gives precisely~\eqref{a1.4X91023}.

To check that condition~\eqref{COND} is satisfied in this case, we point out
that for every~$j\in\{1,\dots,n\}$ we have that~$g_{\eta_j}(\zeta,\eta)=\zeta_j$,
and~$g_{\eta_{n+1}}(\zeta,\eta)=1$. Accordingly,
$$ g_{\eta_j}(\nabla u, Su)=\begin{cases}
u_j & {\mbox{ if }}j\in\{1,\dots,n\},\\
1 & {\mbox{ if }}j=n+1.
\end{cases}$$
Furthermore,
$$ \frac{\partial S^{[j]}u}{\partial x_k}=\begin{cases}\displaystyle
\frac{\partial c_j}{\partial x_k} & {\mbox{ if }}j\in\{1,\dots,n\},\\
\\
h'(u)u_k & {\mbox{ if }}j=n+1.
\end{cases}$$
Consequently, by~\eqref{RED},
\begin{eqnarray*}
\frac{\mathscr{R}}{2|\nabla u|^2}&=&
\sum^{n}_{k=1}\sum^{n+1}_{j=1}g_{\eta_j}(\nabla u, Su)
\frac{\partial S^{[j]}u}{\partial x_k}u_{k}\\
&=&
\sum^{n}_{k=1}\left[\sum^{n}_{j=1}g_{\eta_j}(\nabla u, Su)
\frac{\partial S^{[j]}u}{\partial x_k}u_{k}
+g_{\eta_{n+1}}(\nabla u, Su)
\frac{\partial S^{[n+1]}u}{\partial x_k}u_{k}
\right]
\\&=&
\sum^{n}_{k=1}\left[\sum^{n}_{j=1}
\frac{\partial c_j}{\partial x_k} u_j\,u_{k}
+h'(u)\,u_k^2
\right]\\&\ge&0.
\end{eqnarray*}
This generalizes the result in~(1.10) of~\cite{MR3049726}
to more general operators.

\noindent{\bf (ii).}  As a further example, one can assume
that
\begin{equation}\label{gmonoeta}
{\mbox{$g=g(\zeta,\eta)$ is nondecreasing in~$\eta$}},\end{equation}
and consider the projection operator
$$ Su(x):=u(x_1,0,\dots,0).$$
In this case,
condition~\eqref{COND} is satisfied by all
solutions which are
nondecreasing in the first direction, since,
by~\eqref{RED},
$$ \frac{\mathscr{R}}{2|\nabla u|^2}=
g_{\eta}( \nabla u,Su)
\frac{\partial S u}{\partial x_1}u_{1}=g_{\eta}( \nabla u,Su)\,v_1u_1\ge0,$$
with~$v_1(x):=u_1(x_1,0,\dots,0)$.

\noindent{\bf (iii).} Another interesting case is when~\eqref{gmonoeta}
holds true
and one considers the integral operator
$$ Su(x):=\int_0^{x_1} u(t,x_2,\dots,x_n)\,dt,$$
and then~\eqref{COND} is satisfied by all nonnegative
solutions which are
nondecreasing in every direction (i.e., $u_i\ge0$ for all~$i\in\{1,\dots,n\}$).

Indeed, in this case we have that
$$ v_k(x):=\int_0^{x_1} u_k(t,x_2,\dots,x_n)\,dt\ge0, \quad k=2,\dots,n$$
and hence, by~\eqref{RED},
\begin{eqnarray*} \frac{\mathscr{R}}{2|\nabla u|^2}&=&
\sum^{n}_{k=1} g_{\eta}(\nabla u,Su)
\frac{\partial S u}{\partial x_k}u_{k}\\
&=& g_{\eta}(\nabla u,Su)\left( uu_1+
\sum^{n}_{k=2} v_ku_{k}\right)\\&\ge&0.
\end{eqnarray*}

\noindent{\bf (iv).} One can also assume~\eqref{gmonoeta}
and take into account the convolution operator
$$ Su(x):=\int_{\R^n} u(x-y)\,K(y)\,dy,$$
with~$K\in C^\infty_0(\R^n,[0,+\infty))$.
In this case,
condition~\eqref{COND} is satisfied by all
solutions which are
nondecreasing in every direction, since
$$ v_k(x):=\int_{\R^n} u_k(x-y)\,K(y)\,dy\ge0,$$
and~\eqref{RED} gives that
$$ \frac{\mathscr{R}}{2|\nabla u|^2}=
\sum^{n}_{k=1}g_{\eta}( \nabla u,Su)
\frac{\partial Su}{\partial x_k}u_{k}=g_{\eta}( \nabla u,Su)\sum^{n}_{k=1}v_ku_k\ge0.
$$

\noindent{\bf (v).} More generally, one can also assume~\eqref{gmonoeta}
and take into account the multi-convolution operator
$$ Su(x):=\int_{(\R^n)^d} u(x-y_1)\,\dots\,u(x-y_d)\,
K(y_1,\dots,y_d)\,dy_1\,\dots\,dy_d,$$
with~$K\in C^\infty_0((\R^n)^d,[0,+\infty))$.
Then,
condition~\eqref{COND} is satisfied by all
solutions which are nonnegative, and
nondecreasing in every direction, since
\begin{eqnarray*} v_k(x)&:=&\partial_{x_k}
\int_{(\R^n)^d} u(x-y_1)\,\dots\,u(x-y_d)\,
K(y_1,\dots,y_d)\,dy_1\,\dots\,dy_d\\
&=&\sum_{h=1}^d
\int_{(\R^n)^d} u_k(x-y_h)\,\left(
\prod_{{1\le j\le d}\atop{j\ne h}} u(x-y_j)\right)\,
K(y_1,\dots,y_d)\,dy_1\,\dots\,dy_d
\\&
\ge&0,\end{eqnarray*}
and hence~\eqref{RED} gives that
$$ \frac{\mathscr{R}}{2|\nabla u|^2}=
\sum^{n}_{k=1}g_{\eta}( \nabla u,Su)
\frac{\partial Su}{\partial x_k}u_{k}=g_{\eta}( \nabla u,Su)\sum^{n}_{k=1}v_ku_k\ge0.
$$

\noindent{\bf (vi).} Another interesting example is given by the equation
\begin{equation}\label{a1.4X91023BIS}\mathrm{div}(\Phi'(|\nabla u|^2)\nabla u)=
g(\nabla u, |u|^{q-1}u)\qquad
{\mbox{ in }}\;\R^n,\end{equation}
with~$q\ge1$ and~$g=g(\zeta_1,\dots,\zeta_n,\eta)$
such that~$g_\eta\ge0$.

In this case, one takes~$N:=n+1$ and~$Su:=|u|^{q-1}u$.
Then
$$ \frac{\partial Su}{\partial x_k}=q |u|^{q-1} u_k,$$
and hence, by~\eqref{RED},
$$ \frac{\mathscr{R}}{2|\nabla u|^2}=
\sum^{n}_{k=1} g_{\eta}(\nabla u, Su)
\frac{\partial S u}{\partial x_k}u_{k}=q\, g_{\eta}(\nabla u, Su)
|u|^{q-1}|\nabla u|^2\ge0.$$
\end{rem}}

{\begin{rem}\label{ESEMPI2}\rm
An interesting example satisfying the structural
assumption in~\eqref{COND} is provided by the equation
\begin{equation}\label{a1.487jmalsdfn}
\Delta u=f(u)+(c\cdot\nabla u)\,h(u),
\end{equation}
with~$c\in\R^n$, $f$, $h\in C^1(\R^n)$ and~$h'\ge0$.
In this case, assumption~\eqref{COND} is fulfilled
if~$u$ is monotone nondecreasing in direction~$c$, i.e.~$c\cdot\nabla u\ge0$.

Indeed, in this case we can take~$\Phi(r):=r$, $N:=n+1$,
$g(\zeta,\eta)=g(\zeta_1,\dots,\zeta_n,\eta):=(c\cdot\zeta)\,\eta$
and~$Su:=h(u)$. Then, the general equation in~\eqref{a1.4}
reduces in this setting to the one in~\eqref{a1.487jmalsdfn}.

We observe that~$g_{\zeta_j}(\zeta,\eta)=c_j\eta$ for all~$j\in\{1,\dots,n\}$
and~$g_\eta(\zeta,\eta)=c\cdot\zeta$.
Moreover, by~\eqref{a2.1}, we see that~$\Lambda(r)=1$.
Consequently, we deduce from~\eqref{a1.13} that
\begin{eqnarray*}
\mathscr{R}&=&- 2f(u)g(\nabla u, Su)|\nabla u|^2+2|\nabla u|^2\sum^{n}_{k=1}
g_{\eta}(\nabla u, Su)\frac{\partial Su}{\partial x_k}u_{k}
+2f(u)|\nabla u|^2
\sum^{n}_{j=1}g_{\zeta_j}(\nabla u,Su)u_j\\
&=&- 2f(u)\,(c\cdot\nabla u)\,h(u)\,|\nabla u|^2+2|\nabla u|^2\sum^{n}_{k=1}
(c\cdot\nabla u)h'(u)\,u^2_{k}
+2f(u)|\nabla u|^2
\sum^{n}_{j=1} c_j u_j\,h(u)\\&=&
2|\nabla u|^2\sum^{n}_{k=1}
(c\cdot\nabla u)h'(u)\,u^2_{k}\\&\ge&0
,\end{eqnarray*}
and thus condition~\eqref{COND} is satisfied in this case as well.
\end{rem}}

Following some classical lines of research in~\cites{MR803255, MR1296785, MR3049726}
one has that pointwise
gradient bounds are often related to classification results,
since attaining the potential gauge
at some point provides a very rigid information
that can completely determine the solution. This is the counterpart
of the fact that particles subject to ordinary differential equations
remain motionless if they start with zero velocity at a potential well.
In our setting,
the corresponding result in this direction goes as follows:

\begin{thm}\label{T1.3}
Let~$u\in W^{1,\infty}(\R^n)$
and let the setting in~\eqref{POTE} hold true.
Assume also that~\eqref{GRAD} is satisfied.

Let~$x_0\in\R^n$ be such that~$u(x_0)=r_0$,
with~$F(r_0)=0$ and~$F'(r_0)=0$.

If~$p>2$ in Assumption A, suppose also that
\begin{equation}\label{Z1.16}
\limsup_{r\to r_0}\frac{|F'(r)|}{|r-r_0|^{p-1}}<+\infty.
\end{equation}
Then~$u$ is constantly equal to~$r_0$.
\end{thm}

We observe that condition \eqref{Z1.16} cannot be dropped: indeed, if~$p>2$ and
$$ \beta > \max\left\{ 2,\frac{p}{p-2}\right\},$$
the function
$$ u: \R^n\mapsto\R,~u(x)=|x|^\beta$$
satisfies
\begin{eqnarray*}&&
{\rm div}\big( |\nabla u|^{p-2}\nabla u\big)=
{\rm div}\big( (\beta|x|^{\beta-1})^{p-2}
|x|^{\beta-2}\beta x\big)=
\beta^{p-1} {\rm div}\big( |x|^{\beta p-\beta-p} x\big)\\
&&\qquad=
\beta^{p-1} (\beta p-\beta-p+n) \,|x|^{\beta p-\beta-p}
=\beta^{p-1} (\beta p-\beta-p+n) \,
|u|^{\frac{\beta p-2\beta-p}{\beta}} u=F'(u),
\end{eqnarray*}
with
$$ F(r):=
\frac{ \beta^{p} (\beta p-\beta-p+n) }{(\beta-1)p}\,
|r|^{\frac{(\beta-1)p}{\beta}} .$$
Notice that in this case $F(u(0))=F(0)=0$, and~$F'(0)=0$,
but~$u$ is not constant, and~\eqref{Z1.16} is violated since
$$ \lim_{r\to0} \frac{F'(r)}{|r|^{p-1}}
= \lim_{r\to0}\frac{ \beta^{p} (\beta p-\beta-p+n) }{\beta}\,
|r|^{-\frac{p}{\beta}} =+\infty.$$

The rest of the paper is organized as follows. In
Section~\ref{090}, we show that Assumptions~A and~B are satisfied
in several cases of interest, by proving
Proposition~\ref{STRUCT}.

Section~\ref{THE} introduces the notion of $P$-function
relative to equation~\eqref{a1.4} and contains the computations
needed to check that such a function satisfies a suitable differential
inequality, possibly in terms of the remainder~$\mathscr{R}$.

Then, the proof of Theorem~\ref{T1.2} is presented in Section~\ref{ST1.2},
while Section~\ref{7yh8328} is devoted to the proofs of
Proposition~\ref{CASE} and Corollary~\ref{CASE2},
and Section~\ref{74d96543} contains the
proof of Theorem~\ref{T1.3}.

\vskip4mm
\par\noindent

\section{Structural assumptions and proof of Proposition~\ref{STRUCT}}\label{090}
\setcounter{equation}{0}
\vskip2mm
\noindent
In this section, we will establish Proposition~\ref{STRUCT}. This will be accomplished in Propositions~\ref{PP0} and~\ref{PP1} (which will give, under suitable structural conditions, the setting in Assumption~A),
and in Propositions~\ref{PP2} and~\ref{PP3} (which will give, under suitable structural conditions, the setting in Assumption~B). The precise computational details go as follows.

\begin{prop}\label{PP0}
Assume~\eqref{pj} and~\eqref{cj} hold true. Suppose also that
\begin{equation}\label{p11}
p_1>1,
\end{equation}
\begin{equation}\label{bj1}
b_1\ge0\qquad{\mbox{ and }}\qquad\mu b_1\le b_k\le\frac{b_1}\mu,\quad{\mbox{ for all }}k\in\{1,\dots,m\},
\end{equation}
for some~$\mu\in(0,1)$.
\par
Let also
\begin{equation}\label{sig1} \sigma\in B_M\setminus\{0\},\end{equation}
for some~$M\ge1$. Then we have that
\begin{equation}\label{1.6}
C_1(\sqrt{b_1}+|\sigma|)^{p_1-2}\le
\Phi'(|\sigma|^2)\le C_2(\sqrt{b_1}+|\sigma|)^{p_1-2},
\end{equation}
where
\begin{eqnarray}
\label{ca1}
&&C_1:= c_1 \left(\frac{1}{2}\right)^{\frac{|p_1-2|}{2}}
\in(0,+\infty)\\{\mbox{ and }}\qquad \label{ca2}&&
C_2:=\left(
\frac{2}{\mu}\right)^{\frac{|p_1-2|}2}
\sum_{k=1}^m  c_k
\left(\frac{b_1}{\mu}+M^2\right)^{\frac{p_k-p_1}{2}}\in(0,+\infty).
\end{eqnarray}
\end{prop}

\begin{proof} By~\eqref{PHI}, we have
\begin{equation}\label{PRE}
\Phi'(r)=\sum_{k=1}^m  c_k(b_k+r)^{\frac{p_k-2}{2}}.
\end{equation}
As a consequence,
\begin{equation}\label{PRE2}
\Phi'(|\sigma|^2)=\sum_{k=1}^m  c_k(b_k+|\sigma|^2)^{\frac{p_k-2}{2}}.
\end{equation}
In addition,
we observe
\begin{equation*}
q_1^2+q_2^2\le(q_1+q_2)^2 \le 2(q_1^2+q_2^2)\mathrm{~for~all~}q_1, q_2\in[0,+\infty),
\end{equation*}
therefore
\begin{equation}\label{BETA}
\beta_0+|\sigma|^2\le (\sqrt{\beta_0}+|\sigma|)^2\le 2(\beta_0+|\sigma|^2),
\end{equation}
for all~$\beta_0\ge0$.
\par
Now, to establish the upper bound in~\eqref{1.6}, we use~\eqref{PRE2} and observe that
\begin{equation}\label{SJM}
\Phi'(|\sigma|^2)\le\sum_{k=1}^m  c_k (b_k+|\sigma|^2)^{\frac{p_1-2}{2}}
\left(\frac{b_1}{\mu}+M^2\right)^{\frac{p_k-p_1}{2}}.
\end{equation}
Now we claim
\begin{equation}\label{CC}
(b_k+|\sigma|^2)^{\frac{p_1-2}{2}}\le \frac{1}{\mu^{\frac{|p_1-2|}2}}(b_1+|\sigma|^2)^{\frac{p_1-2}{2}}.
\end{equation}
Indeed, if~$p_1\ge2$, we recall~\eqref{bj1} and have that
\begin{equation*}
b_k+|\sigma|^2\le \frac{b_1}{\mu}+|\sigma|^2\le\frac1{\mu}(b_1+|\sigma|^2),
\end{equation*}
which gives~\eqref{CC}. If instead~$p_1<2$, we use the inequality
$$ b_k+|\sigma|^2\ge \mu b_1+|\sigma|^2\ge \mu(b_1+|\sigma|^2),$$
and this gives~\eqref{CC} in this case as well.
\par
Then, we insert~\eqref{CC} into~\eqref{SJM} and find that
$$\Phi'(|\sigma|^2)\le
\frac{(b_1+|\sigma|^2)^{\frac{p_1-2}{2}}}{\mu^{\frac{|p_1-2|}2}}\sum_{k=1}^mc_k
\left(\frac{b_1}{\mu}+M^2\right)^{\frac{p_k-p_1}{2}}.$$
This and~\eqref{BETA} give
$$\Phi'(|\sigma|^2)\le\left(
\frac{2}{\mu}\right)^{\frac{|p_1-2|}2}
(\sqrt{b_1}+|\sigma|)^{p_1-2}
\sum_{k=1}^m  c_k
\left(\frac{b_1}{\mu}+M^2\right)^{\frac{p_k-p_1}{2}}.$$
{F}rom this and~\eqref{ca2} we conclude that the upper
bound in~\eqref{1.6} is satisfied, as desired.
\par
Now we check the lower bound in~\eqref{1.6}. For this, by~\eqref{PRE2}
and~\eqref{BETA}, we have
\begin{eqnarray*}
\Phi'(|\sigma|^2)\ge  c_1(b_1+|\sigma|^2)^{\frac{p_1-2}{2}}
\ge c_1 \left(\frac{1}{2}\right)^{\frac{|p_1-2|}{2}}
(\sqrt{b_1}+|\sigma|)^{{p_1-2}}.
\end{eqnarray*}
This and~\eqref{ca1} give the lower bound in~\eqref{1.6}.
\end{proof}

\vskip2mm
\par\noindent

\begin{prop}\label{PP1}
Assume~\eqref{pj} and~\eqref{cj} hold true. Suppose also that~\eqref{p11},
\eqref{bj1}
and~\eqref{sig1} are satisfied.
Then, for every~$\xi\in\R^n$
we have that
\begin{equation}\label{1.7}
C_1(\sqrt{b_1} + |\sigma|)^{p_1-2}|\xi|^2\le
\sum_{i,j=1}^n a_{ij}(\sigma)\xi_i\xi_j\le C_2(\sqrt{b_1} + |\sigma|)^{p_1-2}|\xi|^2,
\end{equation}
where
\begin{eqnarray}
\label{c1}
&&C_1:=c_1\,
\min\{1,p_1-1\}
\left(\frac{\mu}{2}\right)^{\frac{|p_1-2|}{2}} \in(0,+\infty)\\{\mbox{ and }}\qquad \label{c2}&&
C_2:=
\left(\frac{2}{\mu}\right)^{\frac{|p_1-2|}{2}}(p_m+1)
\sum_{k=1}^m c_k
\left(\frac{b_1}{\mu}+M^2\right)^{\frac{p_k-p_1}{2}}
\in(0,+\infty).
\end{eqnarray}
\end{prop}

\begin{proof} First of all, from~\eqref{PRE}, we obtain
\begin{equation}\label{PRE7}
\Phi''(r)=\sum_{k=1}^m  \frac{c_k(p_k-2)}{2}(b_k+r)^{\frac{p_k-4}{2}}.
\end{equation}
Accordingly, we have that
\begin{equation*}
a_{ij}(\sigma)=
\sum_{k=1}^m  \Big[ c_k(p_k-2)(b_k+|\sigma|^2)^{\frac{p_k-4}{2}}\sigma_i\sigma_j+c_k(b_k+|\sigma|^2)^{\frac{p_k-2}{2}}\delta_{ij}\Big],
\end{equation*}
and therefore, for every~$\xi=(\xi_1,\dots,\xi_n)\in\R^n$,
\begin{equation}\label{AIJ}
\sum_{i,j=1}^n a_{ij}(\sigma)\xi_i\xi_j=
\sum_{k=1}^m  \Big[ c_k(p_k-2)(b_k+|\sigma|^2)^{\frac{p_k-4}{2}}(\sigma\cdot\xi)^2+c_k(b_k+|\sigma|^2)^{\frac{p_k-2}{2}}|\xi|^2\Big].
\end{equation}
To prove the upper bound in~\eqref{1.7} we argue as follows.
We exploit~\eqref{pj} to see that
\begin{equation}\label{QUE}\begin{split}
\sum_{i,j=1}^n a_{ij}(\sigma)\xi_i\xi_j\,&\le
\sum_{k=1}^m  \Big[ c_k p_k(b_k+|\sigma|^2)^{\frac{p_k-4}{2}}|\sigma|^2|\xi|^2+
c_k(b_k+|\sigma|^2)^{\frac{p_k-2}{2}}|\xi|^2\Big]\\
&\le
\sum_{k=1}^m  \Big[ c_k p_k(b_k+|\sigma|^2)^{\frac{p_k-2}{2}}|\xi|^2+
c_k(b_k+|\sigma|^2)^{\frac{p_k-2}{2}}|\xi|^2\Big]\\
&=
\sum_{k=1}^m  c_k (p_k+1)(b_k+|\sigma|^2)^{\frac{p_k-2}{2}}|\xi|^2\\
&\le(p_m+1)\sum_{k=1}^m  c_k (b_k+|\sigma|^2)^{\frac{p_k-2}{2}}|\xi|^2
.\end{split}\end{equation}
Furthermore, in view of~\eqref{pj}, \eqref{sig1} and~\eqref{CC}, we see that
$$ (b_k+|\sigma|^2)^{\frac{p_k-2}{2}}=
(b_k+|\sigma|^2)^{\frac{p_1-2}{2}}
(b_k+|\sigma|^2)^{\frac{p_k-p_1}{2}}\le\frac{1}{\mu^{\frac{|p_1-2|}{2}}}
(b_1+|\sigma|^2)^{\frac{p_1-2}{2}}
\left(\frac{b_1}{\mu}+M^2\right)^{\frac{p_k-p_1}{2}}.$$
Consequently, by~\eqref{BETA}, we have
$$ (b_k+|\sigma|^2)^{\frac{p_k-2}{2}}\le
\left(\frac{2}{\mu}\right)^{\frac{|p_1-2|}{2}}
(\sqrt{b_1}+|\sigma|)^{{p_1-2}}
\left(\frac{b_1}{\mu}+M^2\right)^{\frac{p_k-p_1}{2}}.$$
Hence~\eqref{QUE} gives that
$$
\sum_{i,j=1}^n a_{ij}(\sigma)\xi_i\xi_j\le
\left(\frac{2}{\mu}\right)^{\frac{|p_1-2|}{2}}(p_m+1)|\xi|^2
(\sqrt{b_1}+|\sigma|)^{p_1-2}
\sum_{k=1}^m  c_k
\left(\frac{b_1}{\mu}+M^2\right)^{\frac{p_k-p_1}{2}}.
$$
This together with~\eqref{c2} establishes the upper bound in~\eqref{1.7}, and we now deal with
the lower bound in~\eqref{1.7}. To this end, we observe that if~$p_k\le2$, then
\begin{eqnarray*}
&& c_k(p_k-2)(b_k+|\sigma|^2)^{\frac{p_k-4}{2}}(\sigma\cdot\xi)^2+
c_k(b_k+|\sigma|^2)^{\frac{p_k-2}{2}}|\xi|^2\\
&=& -c_k(2-p_k)(b_k+|\sigma|^2)^{\frac{p_k-4}{2}}(\sigma\cdot\xi)^2+
c_k(b_k+|\sigma|^2)^{\frac{p_k-2}{2}}|\xi|^2\\&\ge&
-c_k(2-p_k)(b_k+|\sigma|^2)^{\frac{p_k-4}{2}} |\sigma|^2|\xi|^2+
c_k(b_k+|\sigma|^2)^{\frac{p_k-2}{2}}|\xi|^2
\\&\ge&
-c_k(2-p_k)(b_k+|\sigma|^2)^{\frac{p_k-4}{2}} (b_k+|\sigma|^2)|\xi|^2+
c_k(b_k+|\sigma|^2)^{\frac{p_k-2}{2}}|\xi|^2\\
&=& c_k(p_k-1)(b_k+|\sigma|^2)^{\frac{p_k-2}{2}}|\xi|^2\\
&\ge& c_k(p_1-1)\left(\frac{b_1}{\mu}+|\sigma|^2\right)^{\frac{p_k-2}{2}}|\xi|^2\\
&\ge& c_k\,\mu^{\frac{2-p_k}{2}}(p_1-1)\left(b_1+|\sigma|^2\right)^{\frac{p_k-2}{2}}|\xi|^2,\end{eqnarray*}
thanks to~\eqref{bj1}.

This and~\eqref{AIJ} yield that
\begin{equation*}\begin{split}
\sum_{i,j=1}^n a_{ij}(\sigma)\xi_i\xi_j\;& \ge\;
(p_1-1)\sum_{{1\le k\le m}\atop{p_k\le2}}
c_k\,\mu^{\frac{2-p_k}{2}}( b_1+|\sigma|^2)^{\frac{p_k-2}{2}}|\xi|^2\\&\qquad+
\sum_{{1\le k\le m}\atop{p_k>2}}  \Big[ c_k(p_k-2)(b_k+|\sigma|^2)^{\frac{p_k-4}{2}}(\sigma\cdot\xi)^2+
c_k(b_k+|\sigma|^2)^{\frac{p_k-2}{2}}|\xi|^2\Big]\\
& \ge\;
(p_1-1)\sum_{{1\le k\le m}\atop{p_k\le2}}
c_k\,\mu^{\frac{|p_k-2|}{2}}(b_1+|\sigma|^2)^{\frac{p_k-2}{2}}|\xi|^2+
\sum_{{1\le k\le m}\atop{p_k>2}}
c_k(b_k+|\sigma|^2)^{\frac{p_k-2}{2}}|\xi|^2\\
& \ge\;
(p_1-1)\sum_{{1\le k\le m}\atop{p_k\le2}}
c_k\,\mu^{\frac{|p_k-2|}{2}}(b_1+|\sigma|^2)^{\frac{p_k-2}{2}}|\xi|^2+
\sum_{{1\le k\le m}\atop{p_k>2}}
c_k(\mu b_1+|\sigma|^2)^{\frac{p_k-2}{2}}|\xi|^2\\
& \ge\;
(p_1-1)\sum_{{1\le k\le m}\atop{p_k\le2}}
c_k\,\mu^{\frac{|p_k-2|}{2}}(b_1+|\sigma|^2)^{\frac{p_k-2}{2}}|\xi|^2+
\sum_{{1\le k\le m}\atop{p_k>2}}
c_k\,\mu^{\frac{p_k-2}{2}}(b_1+|\sigma|^2)^{\frac{p_k-2}{2}}|\xi|^2\\
& =\;
(p_1-1)\sum_{{1\le k\le m}\atop{p_k\le2}}
c_k\,\mu^{\frac{|p_k-2|}{2}}(b_1+|\sigma|^2)^{\frac{p_k-2}{2}}|\xi|^2+
\sum_{{1\le k\le m}\atop{p_k>2}}
c_k\,\mu^{\frac{|p_k-2|}{2}}(b_1+|\sigma|^2)^{\frac{p_k-2}{2}}|\xi|^2\\
& \ge\;\min\{1,p_1-1\}\sum_{k=1}^m
c_k\,\mu^{\frac{|p_k-2|}{2}}(b_1+|\sigma|^2)^{\frac{p_k-2}{2}}|\xi|^2\\
& \ge\;\min\{1,p_1-1\}
c_1\,\mu^{\frac{|p_1-2|}{2}}(b_1+|\sigma|^2)^{\frac{p_1-2}{2}}|\xi|^2.
\end{split}\end{equation*}
{F}rom this and~\eqref{BETA} we obtain
$$ \sum_{i,j=1}^n a_{ij}(\sigma)\xi_i\xi_j\ge c_1\,
\min\{1,p_1-1\}
\left(\frac{\mu}{2}\right)^{\frac{|p_1-2|}{2}}
(\sqrt{b_1}+|\sigma|)^{{p_1-2}}|\xi|^2.$$
This gives the lower bound in~\eqref{1.7}, thanks to the setting in~\eqref{c1},
and we stress that~$C_1>0$, in light of~\eqref{p11}.
\end{proof}

\vskip2mm
\par\noindent
\begin{prop}\label{PP2}
Assume~\eqref{pj} and~\eqref{cj} hold true. Suppose also that
\begin{equation}\label{bj2}
\mu \le b_k\le\frac{1}\mu,\quad{\mbox{ for all }}k\in\{1,\dots,m\},
\end{equation}
for some~$\mu\in(0,1)$.
\par
Let also
\begin{equation}\label{sig2} \sigma\in B_M\setminus\{0\},\end{equation}
for some~$M\ge1$. Then
we have that
\begin{equation}\label{1.8} C_1(1+|\sigma|)^{-1}\le\Phi'(|\sigma|^2)\le C_2(1+|\sigma|)^{-1},
\end{equation}
where
\begin{eqnarray}
\label{ccx1}
&&C_1:=c_1 \mu^{\frac{p_1}{2}}\in(0,+\infty) \\{\mbox{ and }}\qquad \label{ccx2}&&
C_2:=\sqrt{\frac{2}{\mu}}
\sum_{k=1}^m  c_k\left(\frac1\mu+M^2\right)^{\frac{p_k-1}{2}}\in(0,+\infty).
\end{eqnarray}
\end{prop}

\begin{proof} We use \eqref{pj}, \eqref{bj2} and~\eqref{sig2} to obtain that
\begin{eqnarray*} &&(b_k+|\sigma|^2)^{\frac{p_k-2}{2}}=
(b_k+|\sigma|^2)^{-\frac{1}{2}}
(b_k+|\sigma|^2)^{\frac{p_k-1}{2}}\le
(\mu+|\sigma|^2)^{-\frac{1}{2}}
\left(\frac1\mu+M^2\right)^{\frac{p_k-1}{2}}\\&&\qquad\le\frac{1}{\sqrt\mu}
(1+|\sigma|^2)^{-\frac{1}{2}}
\left(\frac1\mu+M^2\right)^{\frac{p_k-1}{2}}.\end{eqnarray*}
This and~\eqref{BETA} yield
$$ (b_k+|\sigma|^2)^{\frac{p_k-2}{2}}\le
\sqrt{\frac{2}{\mu}}
(1+|\sigma|)^{-1}
\left(\frac1\mu+M^2\right)^{\frac{p_k-1}{2}}.$$
Plugging this information into~\eqref{PRE2}, we see that
$$ \Phi'(|\sigma|^2)\le\sqrt{\frac{2}{\mu}}
(1+|\sigma|)^{-1}
\sum_{k=1}^m  c_k\left(\frac1\mu+M^2\right)^{\frac{p_k-1}{2}}.$$
This and~\eqref{ccx2} give the upper bound in~\eqref{1.8}.
\par
Furthermore, by~\eqref{PRE2} and~\eqref{bj2}, we have
\begin{eqnarray*}
\Phi'(|\sigma|^2)&=&\sum_{k=1}^m  c_k(b_k+|\sigma|^2)^{\frac{p_k-2}{2}}\\
&\ge& c_1(b_1+|\sigma|^2)^{\frac{p_1-2}{2}}\\
&=& c_1(b_1+|\sigma|^2)^{-\frac{1}{2}}(b_1+|\sigma|^2)^{\frac{p_1-1}{2}}\\
&\ge& c_1\left(\frac1\mu+|\sigma|^2\right)^{-\frac{1}{2}}\mu^{\frac{p_1-1}{2}}\\
&\ge& c_1 \sqrt\mu\,\left(1+|\sigma|^2\right)^{-\frac{1}{2}}\mu^{\frac{p_1-1}{2}}.
\end{eqnarray*}
This and~\eqref{BETA} lead to
$$ \Phi'(|\sigma|^2)\ge c_1 \left(1+|\sigma|\right)^{-1}
\mu^{\frac{p_1}{2}}.$$
Hence, recalling~\eqref{ccx1}, we obtain the lower bound in~\eqref{1.8},
as desired.
\end{proof}

\vskip2mm
\par\noindent
\begin{prop}\label{PP3}
Assume~\eqref{pj} and~\eqref{cj} hold true. Suppose also that~\eqref{bj2} and~\eqref{sig2} are satisfied. Then, for every~$\xi'=(\xi_1,\dots,\xi_{n+1})=(\xi,\xi_{n+1})\in\R^n\times\R$ with~$\xi\cdot\sigma=\xi_{n+1}$,
we have that
\begin{equation}\label{1.9}
C_1(1 + |\sigma|)^{-1}|\xi'|^2\le
\sum_{i,j=1}^n a_{ij}(\sigma)\xi_i\xi_j\le C_2(1 + |\sigma|)^{-1}|\xi'|^2,
\end{equation}
where
\begin{eqnarray}
\label{cc1}
&&C_1:=
\frac{c_1\min\{1,p_1-1\}}{ 1+M^2}\mu^{\frac{p_1}{2}}
\in(0,+\infty) \\{\mbox{ and }}\qquad \label{cc2}&&
C_2:=
\sqrt{\frac{2}{\mu}}
\sum_{k=1}^m c_k(|p_k-2|+1)
\left(\frac1\mu+M^2\right)^{\frac{p_k-1}{2}}\in(0,+\infty).
\end{eqnarray}
\end{prop}

\begin{proof} The argument is a careful modification of that used
in the proof of Proposition~\ref{PP1}, taking into special consideration
the $(n+1)$th component of the vector~$\xi'$.
\par
To prove the upper bound in~\eqref{1.9}, we recall~\eqref{AIJ} and perform the following computation:
\begin{eqnarray*}
\sum_{i,j=1}^n a_{ij}(\sigma)\xi_i\xi_j&\le&
\sum_{k=1}^m  \Big[ c_k|p_k-2|(b_k+|\sigma|^2)^{\frac{p_k-4}{2}}(\sigma\cdot\xi)^2+
c_k(b_k+|\sigma|^2)^{\frac{p_k-2}{2}}|\xi|^2\Big]\\
&\le&
\sum_{k=1}^m c_k(|p_k-2|+1)(b_k+|\sigma|^2)^{\frac{p_k-2}{2}}|\xi|^2
\\
&\le&
\sum_{k=1}^m c_k(|p_k-2|+1)
(b_k+|\sigma|^2)^{-\frac{1}{2}}
\left(\frac1\mu+M^2\right)^{\frac{p_k-1}{2}}|\xi|^2
\\
&\le&
\sum_{k=1}^m c_k(|p_k-2|+1)
(\mu+|\sigma|^2)^{-\frac{1}{2}}
\left(\frac1\mu+M^2\right)^{\frac{p_k-1}{2}}|\xi|^2\\
&\le&\frac1{\sqrt\mu}
\sum_{k=1}^m c_k(|p_k-2|+1)
(1+|\sigma|^2)^{-\frac{1}{2}}
\left(\frac1\mu+M^2\right)^{\frac{p_k-1}{2}}|\xi|^2
\end{eqnarray*}
thanks to~\eqref{bj2} and~\eqref{sig2}. Hence, recalling~\eqref{BETA}, we have
$$ \sum_{i,j=1}^n a_{ij}(\sigma)\xi_i\xi_j\le
\sqrt{\frac{2}{\mu}}
\sum_{k=1}^m c_k(|p_k-2|+1)
(1+|\sigma|)^{-1}
\left(\frac1\mu+M^2\right)^{\frac{p_k-1}{2}}|\xi|^2.$$
This proves the upper bound in~\eqref{1.9}, in light of~\eqref{cc2} and the fact that~$|\xi|\le|\xi'|$.
\par
Now we prove the lower bound in~\eqref{1.9}. For this, we use~\eqref{sig2} to see that
\begin{equation}\label{XIT}
|\xi'|^2=|\xi|^2+|\xi\cdot\sigma|^2\le (1+M^2)|\xi|^2.\end{equation}
Also, if~$p_k\le2$, then
\begin{eqnarray*}&&
c_k(p_k-2)(b_k+|\sigma|^2)^{\frac{p_k-4}{2}}(\sigma\cdot\xi)^2+
c_k(b_k+|\sigma|^2)^{\frac{p_k-2}{2}}|\xi|^2\\
&\ge&
-c_k(2-p_k)(b_k+|\sigma|^2)^{\frac{p_k-4}{2}}|\sigma|^2|\xi|^2+
c_k(b_k+|\sigma|^2)^{\frac{p_k-2}{2}}|\xi|^2\\
&\ge&-c_k(2-p_k)(b_k+|\sigma|^2)^{\frac{p_k-2}{2}}|\xi|^2+
c_k(b_k+|\sigma|^2)^{\frac{p_k-2}{2}}|\xi|^2\\
&=&c_k(p_k-1)(b_k+|\sigma|^2)^{\frac{p_k-2}{2}}|\xi|^2.
\end{eqnarray*}
This and~\eqref{AIJ} give that
\begin{eqnarray*}
\sum_{i,j=1}^n a_{ij}(\sigma)\xi_i\xi_j&\ge&
\sum_{{1\le k\le m}\atop{p_k\le2}}
c_k(p_k-1)(b_k+|\sigma|^2)^{\frac{p_k-2}{2}}|\xi|^2\\&&\qquad
+\sum_{{1\le k\le m}\atop{p_k>2}}
\Big[ c_k(p_k-2)(b_k+|\sigma|^2)^{\frac{p_k-4}{2}}(\sigma\cdot\xi)^2+
c_k(b_k+|\sigma|^2)^{\frac{p_k-2}{2}}|\xi|^2\Big]
\\&\ge&\sum_{{1\le k\le m}\atop{p_k\le2}}
c_k(p_1-1)(b_k+|\sigma|^2)^{\frac{p_k-2}{2}}|\xi|^2
+\sum_{{1\le k\le m}\atop{p_k>2}}
c_k(b_k+|\sigma|^2)^{\frac{p_k-2}{2}}|\xi|^2\\
&\ge& \min\{1,p_1-1\}\,
\sum_{k=1}^m
c_k(b_k+|\sigma|^2)^{\frac{p_k-2}{2}}|\xi|^2\\&\ge&
c_1\min\{1,p_1-1\}\,
(b_1+|\sigma|^2)^{\frac{p_1-2}{2}}|\xi|^2.
\end{eqnarray*}
Hence, in view of~\eqref{XIT}, we get
\begin{eqnarray*}
\sum_{i,j=1}^n a_{ij}(\sigma)\xi_i\xi_j&\ge&
\frac{c_1\min\{1,p_1-1\}}{1+M^2}
(b_1+|\sigma|^2)^{\frac{p_1-2}{2}}|\xi'|^2\\
&\ge& \frac{c_1\min\{1,p_1-1\}}{1+M^2}
(b_1+|\sigma|^2)^{-\frac{1}{2}} \mu^{\frac{p_1-1}{2}}|\xi'|^2\\
&\ge& \frac{c_1\min\{1,p_1-1\}}{1+M^2}
\left(\frac1\mu+|\sigma|^2\right)^{-\frac{1}{2}} \mu^{\frac{p_1-1}{2}}|\xi'|^2\\
&\ge& \frac{c_1\min\{1,p_1-1\}}{1+M^2}
(1+|\sigma|^2)^{-\frac{1}{2}} \mu^{\frac{p_1}{2}}|\xi'|^2.
\end{eqnarray*}
This and~\eqref{BETA} give that
$$ \sum_{i,j=1}^n a_{ij}(\sigma)\xi_i\xi_j\ge
\frac{c_1\min\{1,p_1-1\}}{ 1+M^2}
(1+|\sigma|)^{-1} \mu^{\frac{p_1}{2}}|\xi'|^2,$$
that is the lower bound in~\eqref{1.9}, thanks to~\eqref{cc1}.
\end{proof}

By means of the above conclusions, we are in the position of proving Propostion~\ref{STRUCT}:

\begin{proof} The claim in~(i) of Proposition~\ref{STRUCT} directly
follows from Propositions~\ref{PP0}
and~\ref{PP1}.
Similarly, the claim in~(ii) of Proposition~\ref{STRUCT}
is a consequence of Propositions~\ref{PP2}
and~\ref{PP3}.
\end{proof}
\vskip4mm
\par\noindent

\section{$P$-function computations}\label{THE}
\setcounter{equation}{0}
\vskip2mm
\noindent
The goal of this section is to introduce an appropriate
$P$-function
relative to equation~\eqref{a1.4} and establish a differential inequality for it
(combining this with the Maximum Principle, we will obtain
also the desired gradient bounds). To implement this strategy,
for such a solution~$u$, for all~$x\in\R^n$ we define
\begin{equation}\label{a2.3}
P(u;x):=2\Phi'(|\nabla u(x)|^2)|\nabla u(x)|^2-\Phi(|\nabla u(x)|^2)-2F(u(x)),
\end{equation}
and we prove the following result:

\begin{lem}\label{L1.1}
Let $\Omega$ be an open subset of $\mathbb{R}^n$.
Let $u$ be a solution of~\eqref{a1.4} in $\Omega$,
with~$\nabla u\neq0$ in $\Omega$, and
\begin{equation}\label{LAB}
{\mbox{$\Lambda(r)>0$ for all~$r>0$.}}\end{equation}
Let
\begin{equation}\label{a2.2}
d_{ij}(\sigma):=\frac{a_{ij}(\sigma)}{\Lambda(|\sigma|^2)}
\end{equation}
and
\begin{equation}\label{Bi}
B_i(x)=B_i(u;x):=-2\frac{f(u)}{\Lambda(|\nabla u|^2)}\bigg(1+\frac{|\nabla u|^2\Phi''(|\nabla u|^2)}{\Phi'(|\nabla u|^2)}\bigg)\frac{\partial u}{\partial x_i}-\frac{|\nabla u|^2}{\Lambda(|\nabla u|^2)}g_{\zeta_i}(\nabla u,Su).
\end{equation}
Then, we have that
\begin{equation}\label{a2.4}
\sum_{i,j}|\nabla u|^2\frac{\partial}{\partial x_j}\Big(d_{ij}(\nabla u)\frac{\partial P}{\partial x_i}\Big)+\sum_{i}B_i\frac{\partial P}{\partial x_i}\geq\frac{|\nabla P|^2}{2\Lambda(|\nabla u|^2)}+\mathscr{R},~~~~in~\Omega.
\end{equation}
\end{lem}

\begin{proof} By \eqref{LAB}, the map $r\mapsto 2\Phi'(r)r-\Phi(r)$ is invertible,
and we denote by $\Psi$ its inverse. Notice that
\begin{equation}\label{a2.5}
\Psi\big(P(u;x)+2F(u(x))\big)=|\nabla u(x)|^2.
\end{equation}
Moreover, by the definition of $\Psi$ and~\eqref{a2.1}, we have
\begin{equation}\nonumber
1=\frac{\mathrm{d}}{\mathrm{d}r}\Big(\Psi\big(2\Phi'(r)r-\Phi(r)\big)\Big)=\Psi'\big(2\Phi'(r)r-\Phi(r)\big)\Lambda(r),
\end{equation}
hence
\begin{equation}\label{a2.6}
\Psi'\Big(2\Phi'(|\nabla u|^2)|\nabla u|^2-\Phi(|\nabla u|^2)\Big)=\frac{1}{\Lambda(|\nabla u|^2)}.
\end{equation}

Now, differentiating~\eqref{a2.3} and recalling~\eqref{a2.1}, we see that
\begin{equation}\label{a2.7}
\aligned
\frac{\partial P}{\partial x_i}&=2\Big(2\Phi''(|\nabla u(x)|^2)|\nabla u(x)|^2+\Phi'(|\nabla u(x)|^2)\Big)|\nabla u(x)|\frac{\partial|\nabla u(x)|}{\partial x_i}-2f(u)\frac{\partial u}{\partial x_i}\\
&=2\Lambda(|\nabla u|^2)\sum_{k}\frac{\partial^2u}{\partial x_i\partial x_k}\frac{\partial u}{\partial x_k}-2f(u)\frac{\partial u}{\partial x_i}.
\endaligned
\end{equation}
Hence, recalling~\eqref{a2.2}, we get
\begin{eqnarray}\nonumber
&&\sum_{i,j}\frac{\partial}{\partial x_j}\Big(d_{ij}(\nabla u)\frac{\partial P}{\partial x_i}\Big)\\\nonumber
&=&\sum_{i,j}\frac{\partial}{\partial x_j}\Big(-2f(u)d_{ij}(\nabla u)\frac{\partial u}{\partial x_i}+2\Lambda(|\nabla u|^2)\frac{a_{ij}(\nabla u)}{\Lambda(|\nabla u|^2)}\sum_{k}\frac{\partial^2u}{\partial x_i\partial x_k}\frac{\partial u}{\partial x_k}\Big)\\\label{a2.8}
&=&-2\sum_{i,j}\frac{\partial}{\partial x_j}\Big(f(u)d_{ij}(\nabla u)\frac{\partial u}{\partial_{x_i}}\Big)+2\sum_{i,j,k}\frac{\partial}{\partial x_j}\Big(a_{ij}(\nabla u)\frac{\partial^2u}{\partial x_i\partial x_k}\Big)\frac{\partial u}{\partial x_{k}}
\\&&+2\sum_{i,j,k}a_{ij}(\nabla u)\frac{\partial^2u}{\partial x_i\partial x_k}\frac{\partial^2 u}{\partial x_{j}\partial x_{k}}.\nonumber
\end{eqnarray}
Also,~\eqref{a1.5} gives that
\begin{equation}\label{a2.9}
\frac{\partial a_{ij}}{\partial\sigma_l}(\sigma)=\frac{\partial a_{lj}}{\partial\sigma_i}(\sigma).
\end{equation}
By~\eqref{a1.4}, we obtain
\begin{equation}\label{a2.10}
\sum_{i,j}a_{ij}(\nabla u)\frac{\partial^2 u}{\partial x_i\partial x_j}=f(u)+g(\nabla u, Su).
\end{equation}
Therefore, by~\eqref{a2.9} and~\eqref{a2.10}, for any fixed $k$, we have
\begin{eqnarray}\nonumber
&&\sum_{i,j}\frac{\partial}{\partial x_j}\Big(a_{ij}(\nabla u)\frac{\partial^2u}{\partial x_i\partial x_k}\Big)\\\nonumber
&=&\sum_{i,j}\Big(\frac{\partial a_{ij}(\nabla u)}{\partial x_j}\frac{\partial^2u}{\partial x_i\partial x_k}+a_{ij}(\nabla u)\frac{\partial^3u}{\partial x_i\partial x_k\partial x_j}\Big)\\\label{a2.12}
&=&\sum_{i,j}\frac{\partial}{\partial x_k}\Big(a_{ij}(\nabla u)\frac{\partial^2u}{\partial x_i\partial x_j}\Big)\\
&=&f'(u)\frac{\partial u}{\partial x_k}+\sum^{n}_{j=1}g_{\zeta_j}(\nabla u, Su)\frac{\partial^2 u}{\partial x_j\partial x_k}+\sum^{N-n}_{j=1}g_{\eta_j}(\nabla u, Su)\frac{\partial S^{[j]}u}{\partial x_k}.\nonumber
\end{eqnarray}
{F}rom~\eqref{a2.8} and~\eqref{a2.12}, we find that
\begin{eqnarray}\nonumber
&&\sum_{i,j}\frac{\partial}{\partial x_j}\Big(d_{ij}(\nabla u)\frac{\partial P}{\partial x_i}\Big)\\\label{a2.13}
&=&-2\sum_{i,j}f'(u)d_{ij}(\nabla u)\frac{\partial u}{\partial x_i}\frac{\partial u}{\partial x_j}-2f(u)\sum_{i,j}\frac{\partial}{\partial x_j}\Big(d_{ij}(\nabla u)\frac{\partial u}{\partial x_i}\Big)\\\nonumber
&&+2f'(u)\sum_{k}\frac{\partial u}{\partial x_k}\frac{\partial u}{\partial x_{k}}+2\sum_{k}\Bigg[\sum^{n}_{j=1}g_{\zeta_j}(\nabla u, Su)\frac{\partial^2 u}{\partial x_{j}\partial x_{k}}+\sum^{N-n}_{j=1}g_{\eta_j}(\nabla u, Su)\frac{\partial S^{[j]}u}{\partial x_k}\Bigg]\frac{\partial u}{\partial x_{k}}
\\&&+2\sum_{i,j,k}a_{ij}(\nabla u)\frac{\partial^2u}{\partial x_i\partial x_k}\frac{\partial^2 u}{\partial x_{j}\partial x_{k}}.\nonumber
\end{eqnarray}
Furthermore, from~\eqref{a2.1} and~\eqref{a2.2}, we obtain
\begin{eqnarray}\nonumber
&&-\sum_{i,j}f'(u)d_{ij}(\nabla u)\frac{\partial u}{\partial x_i}\frac{\partial u}{\partial x_j}+f'(u)\sum_{k}\frac{\partial u}{\partial x_k}\frac{\partial u}{\partial x_{k}}\\\nonumber
&=&-\sum_{i,j}\frac{f'(u)}{\Lambda(|\nabla u|^2)}\Big[2\Phi''(|\nabla u|^2)\Big(\frac{\partial u}{\partial x_i}\Big)^2\Big(\frac{\partial u}{\partial x_j}\Big)^2+\Phi'(|\nabla u|^2)\delta_{ij}\frac{\partial u}{\partial x_i}\frac{\partial u}{\partial x_j}\Big]\\\label{a2.15}
&&+f'(u)\sum_{k}\Big(\frac{\partial u}{\partial x_k}\Big)^2\\\nonumber
&=&-\frac{f'(u)}{\Lambda(|\nabla u|^2)}\Big[2\Phi''(|\nabla u|^2)|\nabla u|^4+\Phi'(|\nabla u|^2)|\nabla u|^2\Big]
+f'(u)|\nabla u|^2\\\nonumber
&=0.
\end{eqnarray}
Plugging this into~\eqref{a2.13}, we conclude that
\begin{eqnarray}\nonumber
&&\sum_{i,j}\frac{\partial}{\partial x_j}\Big(d_{ij}(\nabla u)\frac{\partial P}{\partial x_i}\Big)\\
&=&-2f(u)\sum_{i,j}\frac{\partial}{\partial x_j}\Big(d_{ij}(\nabla u)\frac{\partial u}{\partial_{x_i}}\Big)
+2\sum_{k}\Bigg[\sum^{n}_{j=1}g_{\zeta_j}(\nabla u, Su)u_{jk}\\\label{a2.16}
&&+\sum^{N-n}_{j=1}g_{\eta_j}(\nabla u, Su)\frac{\partial S^{[j]}u}{\partial x_k}\Bigg]\frac{\partial u}{\partial x_{k}}+2\sum_{i,j,k}a_{ij}(\nabla u)\frac{\partial^2u}{\partial x_i\partial x_k}\frac{\partial^2 u}{\partial x_{j}\partial x_{k}}.\nonumber
\end{eqnarray}
Also, it follows from~\eqref{a2.2} and~\eqref{a2.10} that
\begin{equation}\nonumber
\sum_{i,j}d_{ij}(\nabla u)\frac{\partial^2 u}{\partial x_i\partial x_j}=\frac{f(u)+g(\nabla u, Su)}{\Lambda(|\nabla u|^2)},
\end{equation}
and so~\eqref{a2.16} becomes
\begin{eqnarray}\nonumber
&&\sum_{i,j}\frac{\partial}{\partial x_j}\Big(d_{ij}(\nabla u)\frac{\partial P}{\partial x_i}\Big)\\\nonumber
&=&-2f(u)\sum_{i,j}\frac{\partial}{\partial x_j}d_{ij}(\nabla u)\frac{\partial u}{\partial x_i}-\frac{2f(u)[f(u)+g(\nabla u, Su)]}{\Lambda(|\nabla u|^2)}\\\label{a2.18}
&&+2\sum_{k}\Bigg[\sum^{n}_{j=1}g_{\zeta_j}(\nabla u, Su)u_{jk}+\sum^{N-n}_{j=1}g_{\eta_j}(\nabla u, Su)\frac{\partial S^{[j]}u}{\partial x_k}\Bigg]\frac{\partial u}{\partial x_{k}}\\\nonumber
&&+2\sum_{i,j,k}a_{ij}(\nabla u)\frac{\partial^2u}{\partial x_i\partial x_k}\frac{\partial^2 u}{\partial x_{j}\partial x_{k}}.\nonumber
\end{eqnarray}
Moreover, making use of~\eqref{a1.5},~\eqref{a2.1} and~\eqref{a2.2}, we obtain
\begin{eqnarray}\nonumber
&&\sum_{i,j}\frac{\partial}{\partial x_j}d_{ij}(\nabla u)\frac{\partial u}{\partial x_i}\\\nonumber
&=&\sum_{i,j}\frac{\partial u}{\partial x_i}\frac{\partial}{\partial x_j}\frac{2\Phi''(|\nabla u|^2)\frac{\partial u}{\partial x_i}\frac{\partial u}{\partial x_j}+\Phi'(|\nabla u|^2)\delta_{ij}}{2\Phi''(|\nabla u|^2)|\nabla u|^2+\Phi'(|\nabla u|^2)}\\\label{a2.19}
&&-\sum_{i,j}\frac{\partial u}{\partial x_i}\frac{\big[2\Phi''\frac{\partial u}{\partial x_i}\frac{\partial u}{\partial x_j}+\Phi' \delta_{ij}\big]\big[4\Phi'''|\nabla u|^2+4\Phi''+2\Phi''\big]\sum_{k}\frac{\partial u}{\partial x_k}\frac{\partial^2 u}{\partial x_k\partial x_j}}
{\big[2\Phi''(|\nabla u|^2)|\nabla u|^2+\Phi'(|\nabla u|^2)\big]^2}\\\nonumber
&=&\frac{2\Phi''\big[\Sigma_j|\nabla u|^2\frac{\partial^2 u}{\partial x^2_j}-\sum_{i,j}\frac{\partial u}{\partial x_i}\frac{\partial^2 u}{\partial x_i\partial x_j}\frac{\partial u}{\partial x_j}\big]}{2\Phi''(|\nabla u|^2)|\nabla u|^2+\Phi'(|\nabla u|^2)}\\\nonumber
&=&\frac{2\Phi''(|\nabla u|^2)}{\Lambda(|\nabla u|^2)}\Big(|\nabla u|^2\Delta u-\sum_{i,j}\frac{\partial^2 u}{\partial x_i\partial x_j}\frac{\partial u}{\partial x_i}\frac{\partial u}{\partial x_j}\Big).\nonumber
\end{eqnarray}

Also, from~\eqref{a1.5} and~\eqref{a2.10}, we get
\begin{eqnarray*}
f(u)+g(\nabla u, Su)&=&\sum_{i,j}\Big(2\Phi''(|\nabla u|^2)\frac{\partial u}{\partial x_i}\frac{\partial u}{\partial x_j}+\Phi'(|\nabla u|^2)\delta_{ij}\Big)\frac{\partial^2 u}{\partial x_i\partial x_j}\\
&=&2\Phi''(|\nabla u|^2)\sum_{i,j}\frac{\partial^2 u}{\partial x_i\partial x_j}\frac{\partial u}{\partial x_i}\frac{\partial u}{\partial x_j}+\Phi'(|\nabla u|^2)\Delta u,
\end{eqnarray*}
from which we obtain
\begin{equation}\nonumber
\Delta u=\frac{f(u)+g(\nabla u, Su)}{\Phi'(|\nabla u|^2)}-2\frac{\Phi''(|\nabla u|^2)}{\Phi'(|\nabla u|^2)}\sum_{i,j}\frac{\partial^2 u}{\partial x_i\partial x_j}\frac{\partial u}{\partial x_i}\frac{\partial u}{\partial x_j}.
\end{equation}

Therefore, recalling also~\eqref{a2.7}, we write~\eqref{a2.19} as
\begin{eqnarray}\nonumber
\sum_{i,j}\frac{\partial}{\partial x_j}d_{ij}(\nabla u)\frac{\partial u}{\partial x_i}&=&\frac{2\Phi''(|\nabla u|^2)}{\Lambda(|\nabla u|^2)\Phi'(|\nabla u|^2)}\Big[|\nabla u|^2(f+g)-\Lambda(|\nabla u|^2)\sum_{i,j}\frac{\partial^2 u}{\partial x_i\partial x_j}\frac{\partial u}{\partial x_i}\frac{\partial u}{\partial x_j}\Big]\\\label{a2.20}
&=&\frac{-\Phi''(|\nabla u|^2)}{\Lambda(|\nabla u|^2)\Phi'(|\nabla u|^2)}\sum_{i}\frac{\partial P}{\partial x_i}\frac{\partial u}{\partial x_i}+\frac{2g\Phi''(|\nabla u|^2)|\nabla u|^2}{\Lambda(|\nabla u|^2)\Phi'(|\nabla u|^2)}.
\end{eqnarray}
Thus, exploiting~\eqref{a2.18}, one has
\begin{eqnarray}\nonumber
&&\sum_{i,j}\frac{\partial}{\partial x_j}\Big(d_{ij}(\nabla u)\frac{\partial P}{\partial x_i}\Big)\\\nonumber
&=&\frac{2f(u)\Phi''(|\nabla u|^2)}{\Lambda(|\nabla u|^2)\Phi'(|\nabla u|^2)}\sum_{i}\frac{\partial P}{\partial x_i}\frac{\partial u}{\partial x_i}-\frac{4fg\Phi''(|\nabla u|^2)|\nabla u|^2}{\Lambda(|\nabla u|^2)\Phi'(|\nabla u|^2)}-\frac{2f(u)[f(u)+g(\nabla u, Su)]}{\Lambda(|\nabla u|^2)}\\\label{a2.21}
&&+2\sum_{k}\Bigg[\sum^{n}_{j=1}g_{\zeta_j}(\nabla u, Su)u_{jk}+\sum^{N-n}_{j=1}g_{\eta_j}(\nabla u, Su)\frac{\partial S^{[j]}u}{\partial x_k}\Bigg]\frac{\partial u}{\partial x_{k}}\\\nonumber
&&+2\sum_{i,j,k}a_{ij}(\nabla u)\frac{\partial^2u}{\partial x_i\partial x_k}\frac{\partial^2 u}{\partial x_{j}\partial x_{k}}.\nonumber
\end{eqnarray}

Now we set
\begin{equation}\nonumber
z_{k}=\sum_{i}\frac{\partial^2 u}{\partial x_i\partial x_k}\frac{\partial u}{\partial x_i},
\end{equation}
and we use Schwarz Inequality to see that
\begin{equation}\nonumber
|z_{k}|\leq\sqrt{\sum_{i}\Big(\frac{\partial^2 u}{\partial x_i\partial x_k}\Big)^2}\sqrt{\sum_{i}\Big(\frac{\partial u}{\partial x_i}\Big)^2},
\end{equation}
and so
\begin{eqnarray*}
\sum_{i,j,k}\frac{\partial^2 u}{\partial x_i\partial x_k}\frac{\partial u}{\partial x_i}\frac{\partial^2 u}{\partial x_j\partial x_k}\frac{\partial u}{\partial x_j}&=&\sum_{k}|z_{k}|^2\leq\sum_{k}\bigg(\sum_{i}\Big(\frac{\partial^2 u}{\partial x_i\partial x_k}\Big)^2\bigg)\bigg(\sum_{i}\Big(\frac{\partial u}{\partial x_i}\Big)^2\bigg)\\
&=&|\nabla u|^2\sum_{i,k}\Big(\frac{\partial^2 u}{\partial x_i\partial x_k}\Big)^2.
\end{eqnarray*}
This and~\eqref{a1.5} give that
\begin{eqnarray}\nonumber
&&\sum_{i,j,k}a_{ij}(\nabla u)\frac{\partial^2u}{\partial x_i\partial x_k}\frac{\partial^2 u}{\partial x_{j}\partial x_{k}}\\\nonumber
&=&\sum_{i,j,k}2\Phi''(|\nabla u|^2)\frac{\partial^2 u}{\partial x_i\partial x_k}\frac{\partial^2 u}{\partial x_j\partial x_k}\frac{\partial u}{\partial x_i}\frac{\partial u}{\partial x_j}+\sum_{i,j,k}\Phi'(|\nabla u|^2)\frac{\partial^2 u}{\partial x_i\partial x_k}\frac{\partial^2 u}{\partial x_j\partial x_k}\delta_{ij}\\\label{a2.22}
&\geq&\sum_{i,j,k}2\Phi''(|\nabla u|^2)\frac{\partial^2 u}{\partial x_i\partial x_k}\frac{\partial u}{\partial x_i}\frac{\partial^2 u}{\partial x_j\partial x_k}\frac{\partial u}{\partial x_j}+\sum_{i,j,k}\frac{\Phi'(|\nabla u|^2)}{|\nabla u|^2}\frac{\partial^2 u}{\partial x_i\partial x_k}\frac{\partial u}{\partial x_i}\frac{\partial^2 u}{\partial x_j\partial x_k}\frac{\partial u}{\partial x_j}\\\nonumber
&=&\frac{\Lambda(|\nabla u|^2)}{|\nabla u|^2}\sum_{i,j,k}\frac{\partial^2 u}{\partial x_i\partial x_k}\frac{\partial u}{\partial x_i}\frac{\partial^2 u}{\partial x_j\partial x_k}\frac{\partial u}{\partial x_j}.\nonumber
\end{eqnarray}
Moreover, by~\eqref{a2.7}, we have that
\begin{equation*}
\sum_{i,j,k}\frac{\partial^2 u}{\partial x_i\partial x_k}\frac{\partial u}{\partial x_i}\frac{\partial^2 u}{\partial x_j\partial x_k}\frac{\partial u}{\partial x_j}=\frac{1}{4\Lambda^2(|\nabla u|^2)}\sum_{k}\Big(\frac{\partial P}{\partial x_{k}}+2f\frac{\partial u}{\partial x_k}\Big)^2.
\end{equation*}
This and~\eqref{a2.22} lead to
\begin{eqnarray*}
\sum_{i,j,k}a_{ij}(\nabla u)\frac{\partial^2u}{\partial x_i\partial x_k}\frac{\partial^2 u}{\partial x_{j}\partial x_{k}}&\geq&
\frac{1}{4\Lambda(|\nabla u|^2)|\nabla u|^2}\sum_{k}\Big(\frac{\partial P}{\partial x_{k}}+2f\frac{\partial u}{\partial x_k}\Big)^2\\
&=&\frac{|\nabla P|^2}{4\Lambda(|\nabla u|^2)|\nabla u|^2}+\frac{f\sum_{k}\frac{\partial P}{\partial x_k}\frac{\partial u}{\partial x_k}}{\Lambda(|\nabla u|^2)|\nabla u|^2}+\frac{f^2}{\Lambda(|\nabla u|^2)}.
\end{eqnarray*}
By substituting this into~\eqref{a2.21}, we obtain
\begin{eqnarray*}
&&\sum_{i,j}\frac{\partial}{\partial x_j}\Big(d_{ij}(\nabla u)\frac{\partial P}{\partial x_i}\Big)\\
&\geq&\frac{2f(u)\Phi''(|\nabla u|^2)}{\Lambda(|\nabla u|^2)\Phi'(|\nabla u|^2)}\sum_{i}\frac{\partial P}{\partial x_i}\frac{\partial u}{\partial x_i}-\frac{4fg\Phi''(|\nabla u|^2)|\nabla u|^2}{\Lambda(|\nabla u|^2)\Phi'(|\nabla u|^2)}-\frac{2f(u)[f(u)+g(\nabla u, Su)]}{\Lambda(|\nabla u|^2)}\\
&&+2\sum_{k}\Bigg[\sum^{n}_{j=1}g_{\zeta_j}(\nabla u, Su)\frac{\partial^2 u}{\partial x_{j}\partial x_{k}}+\sum^{N-n}_{j=1}g_{\eta_j}(\nabla u, Su)\frac{\partial S^{[j]}u}{\partial x_k}\Bigg]\frac{\partial u}{\partial x_{k}}
+\frac{|\nabla P|^2}{2\Lambda(|\nabla u|^2)|\nabla u|^2}\\
&&+2\frac{f\sum_{k}\frac{\partial P}{\partial x_k}\frac{\partial u}{\partial x_k}}{\Lambda(|\nabla u|^2)|\nabla u|^2}+2\frac{f^2}{\Lambda(|\nabla u|^2)}.
\end{eqnarray*}
Therefore, we have that
\begin{eqnarray}\nonumber
&&\sum_{i,j}\frac{\partial}{\partial x_j}\Big(d_{ij}(\nabla u)\frac{\partial P}{\partial x_i}\Big)-2\frac{f(u)}{\Lambda(|\nabla u|^2)|\nabla u|^2}\bigg(1+\frac{\Phi''(|\nabla u|^2)|\nabla u|^2}{\Phi'(|\nabla u|^2)}\bigg)\sum_{k}\frac{\partial P}{\partial x_k}\frac{\partial u}{\partial x_k}
\\\nonumber
&\geq&-\frac{4fg\Phi''(|\nabla u|^2)|\nabla u|^2}{\Lambda(|\nabla u|^2)\Phi'(|\nabla u|^2)}-\frac{2f(u)g(\nabla u, Su)}{\Lambda(|\nabla u|^2)}\\\label{a2.23}
&&+2\sum_{k}\Bigg[\sum^{n}_{j=1}g_{\zeta_j}(\nabla u, Su)\frac{\partial^2 u}{\partial x_{j}\partial x_{k}}+\sum^{N-n}_{j=1}g_{\eta_j}(\nabla u, Su)\frac{\partial S^{[j]}u}{\partial x_k}\Bigg]\frac{\partial u}{\partial x_{k}}
+\frac{|\nabla P|^2}{2\Lambda(|\nabla u|^2)|\nabla u|^2}\\\nonumber
&=&-\frac{2f(u)g(\nabla u, Su)}{\Phi'(|\nabla u|^2)}+2\sum_{k}\Bigg[\sum^{n}_{j=1}g_{\zeta_j}(\nabla u, Su)\frac{\partial^2 u}{\partial x_{j}\partial x_{k}}+\sum^{N-n}_{j=1}g_{\eta_j}(\nabla u, Su)\frac{\partial S^{[j]}u}{\partial x_k}\Bigg]\frac{\partial u}{\partial x_{k}}\\\nonumber
&&+\frac{|\nabla P|^2}{2\Lambda(|\nabla u|^2)|\nabla u|^2}.\nonumber
\end{eqnarray}

Now, for $j$ fixed, we use~\eqref{a2.5} to get that
\begin{equation}\nonumber
2\sum_{k}\frac{\partial^2 u}{\partial x_j\partial x_k}\frac{\partial u}{\partial x_k}=\frac{\partial|\nabla u(x)|^2}{\partial x_j}=\Psi'(P+2F)\bigg(\frac{\partial P}{\partial x_j}+2f(u)\frac{\partial u}{\partial x_j}\bigg).
\end{equation}
Consequently, by~\eqref{a2.6}, we conclude
\begin{eqnarray*}
2\sum_{k,j}g_{\zeta_j}(\nabla u, Su)\frac{\partial^2 u}{\partial x_{j}\partial x_{k}}\frac{\partial u}{\partial x_{k}}&=\Psi'(P+2F)\sum_{j}g_{\zeta_j}(\nabla u, Su)\bigg(\frac{\partial P}{\partial x_j}+2f(u)\frac{\partial u}{\partial x_j}\bigg)\\
&=\frac{1}{\Lambda(|u(x)|^2)}\sum_{j}g_{\zeta_j}(\nabla u, Su)\bigg(\frac{\partial P}{\partial x_j}+2f(u)\frac{\partial u}{\partial x_j}\bigg).
\end{eqnarray*}
Multiplying both sides of~\eqref{a2.23} by $|\nabla u|^2$, we see that
\begin{eqnarray*}
&&\sum_{i,j}|\nabla u|^2\frac{\partial}{\partial x_j}\Big(d_{ij}(\nabla u)\frac{\partial P}{\partial x_i}\Big)-2\frac{f(u)}{\Lambda(|\nabla u|^2) }\bigg(1+\frac{\Phi''(|\nabla u|^2)|\nabla u|^2}{\Phi'(|\nabla u|^2)}\bigg)\sum_{k}\frac{\partial P}{\partial x_k}\frac{\partial u}{\partial x_k}\\
&\geq&-\frac{2f(u)g(\nabla u, Su)|\nabla u|^2}{\Phi'(|\nabla u|^2)}+\frac{|\nabla P|^2}{2\Lambda(|\nabla u|^2)}\\
&&+2|\nabla u|^2\sum_{k}\Bigg[\sum^{n}_{j=1}g_{\zeta_j}(\nabla u, Su)\frac{\partial^2 u}{\partial x_{j}\partial x_{k}}+\sum^{N-n}_{j=1}g_{\eta_j}(\nabla u, Su)\frac{\partial S^{[j]}u}{\partial x_k}\Bigg]\frac{\partial u}{\partial x_{k}}\\
&=&-\frac{2f(u)g(\nabla u, Su)|\nabla u|^2}{\Phi'(|\nabla u|^2)}+2|\nabla u|^2\sum^{n}_{k=1}\sum^{N-n}_{j=1}g_{\eta_j}(\nabla u, Su)\frac{\partial S^{[j]}u}{\partial x_k}\frac{\partial u}{\partial x_{k}}\\
&&+\frac{|\nabla u|^2}{\Lambda(|u(x)|^2)}\sum_{j}g_{\zeta_j}(\nabla u, Su)\bigg(\frac{\partial P}{\partial x_j}+2f(u)\frac{\partial u}{\partial x_j}\bigg)+\frac{|\nabla P|^2}{2\Lambda(|\nabla u|^2)} \\
&=&\mathscr{R}(x)+\frac{|\nabla u|^2}{\Lambda(|u(x)|^2)}\sum_{j}g_{\zeta_j}(\nabla u, Su)\frac{\partial P}{\partial x_j}+\frac{|\nabla P|^2}{2\Lambda(|\nabla u|^2)}.
\end{eqnarray*}
{F}rom this, we obtain the desired result in~\eqref{a2.4}.
\end{proof}

\vskip4mm
\par\noindent
\section{Proof of Theorem~\ref{T1.2}}\label{ST1.2}
\setcounter{equation}{0}
\vskip2mm
\noindent
This section contains the proof of the pointwise gradient estimate
in~\eqref{GRAD}. This relies on Lemma~\ref{L1.1} and the Maximum Principle.
The technical details go as follows:\smallskip

\begin{proof}[Proof of Theorem~\ref{T1.2}]
First of all, we observe that~\eqref{LAB} holds true.
Indeed, taking~$\xi:=(1,0,\dots,0)$ and~$\sigma:=\sqrt{r}\,\xi$,
we deduce from~\eqref{a1.5} and~\eqref{a2.1} that
\begin{equation}\label{7-IA}
\sum_{i,j=1}^n a_{ij}(\sigma)\xi_i\xi_j=
2\Phi''(|\sigma|^2)(\sigma\cdot\xi)^2+\Phi'(|\sigma|^2)|\xi|^2=
2\Phi''(r)r+\Phi'(r)=\Lambda(r).
\end{equation}
Hence, if Assumption~A is satisfied, we obtain
\begin{equation}\label{jms}
\Lambda(r) =\sum_{i,j=1}^n a_{ij}(\sigma)\xi_i\xi_j
\ge C_1(a+|\sigma|)^{p-2}|\xi|^2=C_1(a+\sqrt{r})^{p-2}>0.\end{equation}
If instead Assumption~B is satisfied, we deduce from~\eqref{7-IA} that
\begin{equation}\label{3232}\begin{split}&
\Lambda(r) =\sum_{i,j=1}^n a_{ij}(\sigma)\xi_i\xi_j
\ge
C_1(1+|\sigma|)^{-1}|\xi'|^2=
C_1(1+|\sigma|)^{-1}\big(|\xi|^2+(\sigma\cdot\xi)^2\big)\\&\qquad\qquad=
C_1(1+\sqrt{r})^{-1} (1+r)>0.\end{split}
\end{equation}
This observation and~\eqref{jms} show that~\eqref{LAB}
is satisfied, and therefore we are in the position of applying
Lemma~\ref{L1.1}. In this way,
recalling~\eqref{COND} and~\eqref{a2.4}, we see
\begin{equation}\label{a3.0}
\sum_{i,j}\frac{\partial}{\partial x_{j}}\bigg(d_{ij}(\nabla u(x))
\frac{\partial P(u;x)}{\partial x_{i}}\bigg)+\frac{B(u;x)\cdot\nabla P(u;x)}{|\nabla u(x)|^2}\geq0
\end{equation}
in $\{\nabla u\neq0\}$, where the notations
in~\eqref{a2.3} and~\eqref{Bi} have been utilized.

{F}rom this, we can
repeat some classical arguments
used also in the proof of Theorem~1.2 in~\cite{MR3049726}
to obtain our Theorem~\ref{T1.2}. We show the arguments in full detail
for the facility of the reader.
Besides, in order to address the general case treated in this paper, these classical arguments need to be carefully adapted, producing a number of additional technical difficulties.

Recalling the notation in~\eqref{a2.3}, we define
\begin{equation}\label{PoP}
P_0:=\sup_{x\in\mathbb{R}^n}P(u;x).
\end{equation}
We claim that
\begin{equation}\label{a3.2}
P_0\leq0.
\end{equation}
To prove this, we assume by contradiction that
\begin{equation}\label{a3.3}
P_0>0.
\end{equation}
First, take sequence $z_{k}\in\mathbb{R}^n$ such that
\begin{equation}\label{333}
\lim_{k\rightarrow+\infty}P(u; z_k)=P_0.
\end{equation}
We can define~$w_k(x)=u(x+z_k)$. This function satisfies an elliptic equation
with bounded right hand side and therefore,
by elliptic regularity theory (possibly reducing Assumption~B
to Assumption~A with~$p=2$), we have that, for every~$R>0$,
\begin{equation}\label{REGO}
\|w_k\|_{C^{1,\gamma}(B_R)}<+\infty,
\end{equation}
for some~$\gamma\in(0,1)$.

Also, from~\eqref{PoP}, we have that
\begin{equation}\label{bYAkdc}
P_0\ge P(w_k;x),\qquad{\mbox{ for all }}x\in\R^n.
\end{equation}
Furthermore, by~\eqref{a2.3}, we get
\begin{eqnarray*}&& P(u;z_k)=2\Phi'(|\nabla u(z_k)|^2)|\nabla u(z_k)|^2
-\Phi(|\nabla u(z_k)|^2)-2F(u(z_k))\\&&\qquad=
2\Phi'(|\nabla w_k(0)|^2)|\nabla w_k(0)|^2
-\Phi(|\nabla w_k(0)|^2)-2F(w_k(0))
=P(w_k;0).\end{eqnarray*}
In view of this and~\eqref{333},
we conclude that
\begin{equation}\label{a3.4}
\lim_{k\rightarrow+\infty}P(w_k;0)=P_0.
\end{equation}
By the Theorem of Ascoli-Arzel\`a (and up to a subsequence) and possibly renaming~$\gamma$,
we may suppose that $w_{k}$ converges to some $w$ in $C^{1,\gamma}_{\mathrm{loc}}(\mathbb{R}^n)$, and therefore, by~\eqref{a2.3}, we see
$$ \lim_{k\rightarrow+\infty}P(w_k;x)=P(w;x),\qquad{\mbox{ for all $x\in\R^n$.}}$$
Using this, \eqref{bYAkdc}
and~\eqref{a3.4}, we thereby obtain
\begin{equation}\label{PoPOAsdK}
P(w;x)=\lim_{k\rightarrow+\infty}P(w_k;x)\le P_0
=\lim_{k\rightarrow+\infty}P(w_k;0)=P(w;0).
\end{equation}
Now, we define
\begin{equation*}
\mathcal{N}:=\{x\in\mathbb{R}^n\mathrm{~s.t.~}P(w;x)=P_0\}.
\end{equation*}
We observe that $0\in\mathcal{N}$, thanks to~\eqref{PoPOAsdK}, and hence
\begin{equation}\label{a3.5}
\mathcal{N}\neq\varnothing.
\end{equation}
Also,
by the continuity of~$P$ and~$w$,
\begin{equation}\label{a3.6}
\mathcal{N}\mathrm{~is~closed.}
\end{equation}
Here, we denote $\Psi^{-1}$ by $\Gamma$ for simplicity, so
\begin{equation}\label{GAMMAR} \Gamma(r):=2\Phi'(r)r-\Phi(r),\end{equation}
and we
claim that for all~$r\in\left[0,\|u\|_{W^{1,\infty}(\mathbb{R}^n)}\right]$
\begin{equation}\label{NONP}
\Gamma(r)\le C\sqrt{r}.
\end{equation}
To prove this, we first remark that
\begin{equation}\label{NONP0}
\Gamma(0)=0,\end{equation}
by $\Phi(0)=0$ and either~\eqref{083-384-1} or~\eqref{083-384-2}. Moreover,
taking~$\sigma:=\sqrt{r}e_1$,
$\xi:=e_1$, $\xi_{n+1}:=\sqrt{r}$, by~\eqref{a1.5} we see that
\begin{eqnarray*}
a_{ij}(\sigma)\xi_i\xi_j&
=&\begin{cases}
2\Phi''(r)r+\Phi'(r)=\Gamma'(r)
&{\mbox{ if }}i=j=1,\\
0&{\mbox{ otherwise,}}
\end{cases}
\end{eqnarray*}
and accordingly
\begin{equation}\label{767682} \sum_{i,j=1}^n a_{ij}(\sigma)\xi_i\xi_j=\Gamma'(r).\end{equation}
Now, to prove~\eqref{NONP} we distinguish two
cases, according to whether Assumption~A
or Assumption~B is satisfied. First of all, if Assumption~A is satisfied, we use~\eqref{a1.7}
and~\eqref{767682}
to see that
\begin{equation} \label{AxTGS}
\Gamma'(r) \le C_2(a+|\sigma|)^{p-2}|\xi|^2=
C_2(a+\sqrt{r})^{p-2}.\end{equation}
We now distinguish two subcases, depending on~$p$.
If~$p\ge2$, we deduce from~\eqref{AxTGS} that
$$ \Gamma'(r)\le C\left(1+r^{\frac{p-2}2}\right),$$
for some~$C>0$. This and~\eqref{NONP0} yield that
\begin{eqnarray*}&& \Gamma(r)=\int_0^r\Gamma'(\rho)\,d\rho\le
C\left(r+\frac{2}{p}r^{\frac{p}{2}}\right)
=C \left(\sqrt{r}+\frac{2}{p}r^{\frac{p-1}{2}}\right)\,\sqrt{r}\\&&\qquad\qquad\le
C \left(\sqrt{\|u\|_{W^{1,\infty}(\mathbb{R}^n)}}+\frac{2}{p}\|u\|_{W^{1,\infty}(\mathbb{R}^n)}^{\frac{p-1}{2}}\right)\,\sqrt{r},
\end{eqnarray*}
for all~$r\in\left[0,\|u\|_{W^{1,\infty}(\mathbb{R}^n)}\right]$,
and this gives~\eqref{NONP}, up to renaming~$C>0$.

On the other hand, if~$p\in(1,2)$,
we deduce from~\eqref{AxTGS} that
$$ \Gamma'(r)\le \frac{C_2}{(a+\sqrt{r})^{2-p}}\le
\frac{C_2}{r^{\frac{2-p}{2}}},$$
which, together with~\eqref{NONP0}, gives that
\begin{eqnarray*}&& \Gamma(r)=\int_0^r\Gamma'(\rho)\,d\rho\le
\frac{2C_2 r^{\frac{p}{2}}}{p}\le
\frac{2C_2 \|u\|_{W^{1,\infty}(\mathbb{R}^n)}^{\frac{p-1}{2}}}{p}\,\sqrt{r},
\end{eqnarray*}
for all~$r\in\left[0,\|u\|_{W^{1,\infty}(\mathbb{R}^n)}\right]$,
and this gives~\eqref{NONP}.

It remains to prove~\eqref{NONP} if Assumption~B holds true.
In this case, we use~\eqref{a1.9}
and~\eqref{767682}
to see that
$$\Gamma'(r)\le C_2(1+|\sigma|)^{-1}|\xi'|^2=
C_2(1+|\sigma|)^{-1}(|\xi|^2+|\xi_{n+1}|^2)=
C_2(1+\sqrt{r})^{-1}(1+r)
\le C_2(1+r).
$$
This and~\eqref{NONP0} give that
$$ \Gamma(r)\le C_2\left(r+\frac{r^2}2\right)\le
C_2\left(\sqrt{\|u\|_{W^{1,\infty}(\mathbb{R}^n)}}+\frac{\|u\|_{W^{1,\infty}(\mathbb{R}^n)}^{\frac32}}2\right)\,\sqrt{r}
$$
for all~$r\in\left[0,\|u\|_{W^{1,\infty}(\mathbb{R}^n)}\right]$,
and the proof of~\eqref{NONP} is thereby complete.

Now, we claim that
\begin{equation}\label{a3.7}
\mathcal{N}\mathrm{~is~open.}
\end{equation}
For this, let $y_0\in \mathcal{N}$.
We recall that~$F\ge0$ on the range of~$u$, thanks to~\eqref{POTE}.
Then, in light of~\eqref{a2.3},
\eqref{GAMMAR}
and~\eqref{NONP} we see that
\begin{equation}\label{yHShsISUSsddfU}\begin{split}&
P_0=P(w;y_0)=
2\Phi'(|\nabla w(y_0)|^2)|\nabla w(y_0)|^2-\Phi(|\nabla w(y_0)|^2)-2F(w(y_0))
\\
&\qquad=\Gamma(|\nabla w(y_0)|^2)-2F(w(y_0))\le\Gamma(|\nabla w(y_0)|^2)\le C|\nabla w(y_0)|.
\end{split}\end{equation}
Now, we set
$$ \kappa:=\frac{P_0}{2C},$$
and,
recalling~\eqref{a3.3}, we observe that~$\kappa>0$.
As a consequence,
in light of~\eqref{yHShsISUSsddfU}, it follows that
there exists $\varrho>0$ such that
\begin{equation*}
|\nabla w(x)|\geq\kappa,\qquad{\mbox{ for any $x\in B_{\varrho}(y_0)$}}.\end{equation*}
Therefore, there exists~$\bar k\in\mathbb{N}$ such that
for all~$k\ge\bar k$ and all~$x\in B_{\varrho}(y_0)$ we have that
\begin{equation}\label{gradb} |\nabla u(x+z_k)|=|\nabla w_k(x)|\geq\frac\kappa2.\end{equation}
In particular, $\nabla u(x+z_k)\ne0$ and therefore, by~\eqref{a3.0}, we get
\begin{equation}\label{a3.0BOSX}
\sum_{i,j}\frac{\partial}{\partial x_{j}}\bigg(d_{ij}(\nabla u(x+z_k))
\frac{\partial P(u;x+z_k)}{\partial x_{i}}\bigg)+\frac{B(u;x+z_k)\cdot\nabla P(u;x+z_k)}{|\nabla u(x+z_k)|^2}\geq0
.\end{equation}
Moreover, by~\eqref{a2.3}, we have~$P(u;x+z_k)=P(w_k;x)$,
and therefore we can write~\eqref{a3.0BOSX}
in the form
\begin{equation}\label{a3.0BOS}
\begin{split}0\,&\le\,
\sum_{i,j}\frac{\partial}{\partial x_{j}}\bigg(d_{ij}(\nabla w_k(x))
\frac{\partial P(w_k;x)}{\partial x_{i}}\bigg)+\frac{B(u;x+z_k)\cdot\nabla P(w_k;x)}{|\nabla u(x+z_k)|^2}
\\&=\,
\sum_{i,j}\frac{\partial}{\partial x_{j}}\bigg(d_{ij}(\nabla w_k(x))
\frac{\partial P(w_k;x)}{\partial x_{i}}\bigg)+
\beta_k(x)\cdot\nabla P(w_k;x),\end{split}\end{equation}
for all~$x\in B_{\varrho}(y_0)$,
as long as~$k\ge\bar{k}$,
where
\begin{equation} \label{betak}\beta_k(x):=\frac{B(u;x+z_k)}{|\nabla u(x+z_k)|^2}.\end{equation}
We stress that, by~\eqref{Bi}, \eqref{gradb} and~\eqref{betak},
we can obtain
\begin{eqnarray*}&&\sup_{k\ge\bar{k}}
\|\beta_k\|_{L^\infty(B_{\varrho}(y_0),\mathbb{R}^n)}\\&\le&\frac4{\kappa^2} \Bigg[
\frac{2\|f\|_{L^\infty(u({\mathbb{R}}^n))}}{\displaystyle
\inf_{ \kappa/2\le|\zeta|\le \|u\|_{W^{1,\infty}(\mathbb{R}^n)}}
\Lambda(|\zeta|^2)}\bigg(1+
\sup_{\kappa/2\le |\zeta|\le \|u\|_{W^{1,\infty}(\mathbb{R}^n)}}
\frac{|\zeta|^2\;\Phi''(|\zeta|^2)}{\Phi'(|\zeta|^2)}\bigg)
\,\|u\|_{W^{1,\infty}(\mathbb{R}^n)}\\&&\qquad
+
\frac{\|u\|_{W^{1,\infty}(\mathbb{R}^n)}^2}{
\displaystyle
\inf_{\kappa/2\le|\zeta|\le \|u\|_{W^{1,\infty}(\mathbb{R}^n)}}
\Lambda(|\zeta|^2)}
\sup_{{j\in\{1,\dots,n\}}\atop{{(\zeta, \eta)\in \mathbb{R}^n\times\mathbb{R}^{N-n}}
\atop{|\zeta|\le \|u\|_{W^{1,\infty}(\mathbb{R}^n)}}} }\big|
g_{\zeta_j}(\zeta,\eta)
\big|\Bigg],
\end{eqnarray*}
which is bounded, thanks to~\eqref{SOMA}
(recall also~\eqref{jms} and~\eqref{3232}).
Therefore, up to subsequences,
we can suppose that
\begin{equation}\label{a3.0BO3}\begin{split}&
{\mbox{$\beta_k$ converges to some~$\beta\in L^\infty(B_{\varrho}(y_0),\mathbb{R}^n)$}}\\&
{\mbox{weakly in~$L^2(B_{\varrho}(y_0),\mathbb{R}^n)$
and weakly-$*$ in $L^\infty(B_{\varrho}(y_0),\mathbb{R}^n)$.}}\end{split}\end{equation}
Furthermore, by~\eqref{a1.4}, we conclude
\begin{eqnarray*}
0 &=& \mathrm{div}(\Phi'(|\nabla u(x+z_k)|^2)\nabla u(x+z_k))-
f(u(x+z_k))-g(\nabla u(x+z_k), Su(x+z_k))\\
&=& \mathrm{div}(\Phi'(|\nabla w_k(x)|^2)\nabla w_k(x))-
f(w_k(x))-g(\nabla w_k(x), Su(x+z_k))\\
&=&
\mathrm{div}(\Phi'(|\nabla w_k(x)|^2)\nabla w_k(x))-
\tilde f_k(x),
\end{eqnarray*}
where
$$\tilde f_k(x):=
f(w_k(x))+g(\nabla w_k(x), Su(x+z_k)).$$
In view of~\eqref{CONDalpha}
and~\eqref{REGO}, we have that~$\tilde f_k\in C^{0,\gamma}( B_{\varrho}(y_0))$,
with
$$ \sup_{k\ge\bar k}\|\tilde f_k\|_{ C^{0,\gamma}( B_{\varrho}(y_0))}<+\infty.$$
Consequently,
by~\eqref{gradb} and uniform elliptic regularity theory,
we obtain that
$$ \sup_{k\ge\bar k}\|w_k\|_{ C^{2,\gamma}( B_{\varrho}(y_0))}<+\infty.$$
Therefore, up to a subsequence and possibly
renaming~$\gamma$, we can suppose that~$w_k$ converges
to~$w$ in~$C^{2,\gamma}( B_{\varrho}(y_0))$, as $k\to+\infty$.

As a consequence, recalling~\eqref{a2.3},
we conclude that
\begin{equation}\label{a3.0BO2}
{\mbox{$\nabla P(w_k;\cdot)$ converges to
$\nabla P(w;\cdot)$ in~$C^{0,\gamma}( B_{\varrho}(y_0),\mathbb{R}^n)$, as $k\to+\infty$.}}
\end{equation}
Exploiting~\eqref{a3.0BOS}, \eqref{a3.0BO3} and~\eqref{a3.0BO2},
we obtain that
$$ 0\le
\sum_{i,j}\frac{\partial}{\partial x_{j}}\bigg(d_{ij}(\nabla w(x))
\frac{\partial P(w;x)}{\partial x_{i}}\bigg)+
\beta(x)\cdot\nabla P(w;x),$$
for all~$x\in B_{\varrho}(y_0)$, in the distributional sense.

Therefore, recalling~\eqref{PoPOAsdK},
by Maximum Principle (see e.g.~\cite[Theorem 8.19]{MR038022},
or~\cite{MR2356201}),
it follows that $P(w; x)=P_0$ for any $x\in B_{\varrho}(y_0)$, and
this establishes~\eqref{a3.7}.

Now, by~\eqref{a3.5}
and~\eqref{a3.7}, we infer that $\mathcal{N}$ is both closed and open, so that $\mathcal{N}=\mathbb{R}^n$, that is
\begin{equation}\label{a3.8}
P(w;x)=P_0\mathrm{~for~any~}x\in\mathbb{R}^n.
\end{equation}
On the other hand, since $w$ is bounded, by following the gradient lines we find a sequence of points $\tau_{j}$ such that
\begin{equation*}
\lim_{j\rightarrow+\infty}\nabla w(\tau_{j})=0.
\end{equation*}
By using this in~\eqref{a3.8}, we obtain
\begin{equation*}
0\geq\limsup_{j\rightarrow+\infty}-2F(w(\tau_j))=\limsup_{j\rightarrow+\infty}P(w;\tau_j)=P_0,
\end{equation*}
which is in contradiction with~\eqref{a3.3}.
This proves~\eqref{a3.2}, from which Theorem~\ref{T1.2} follows at once.
\end{proof}

\vskip4mm
\par\noindent
\section{Proofs of Proposition \ref{CASE} and Corollary~\ref{CASE2}}\label{7yh8328}
\setcounter{equation}{0}
\vskip2mm
\noindent
We start by proving Proposition \ref{CASE}:

\begin{proof}[Proof of Proposition \ref{CASE}]
{F}rom the definition of $\Phi(r)$ in~\eqref{PHI} and $\Lambda$ in~\eqref{a2.1} (recall also~\eqref{PRE}
and~\eqref{PRE7}) we have that
\begin{equation*}
\begin{split}&
\Lambda(r)=2r\Phi''(r)+\Phi'(r)=
\sum_{k=1}^{m}\Big[c_k(p_k-2)r(b_k+r)^{\frac{p_k-4}{2}}+c_k(b_k+r)^{\frac{p_k-2}{2}}\Big]
\\&\qquad\qquad=
\sum_{k=1}^{m}c_k(b_k+r)^{\frac{p_k-4}{2}}\big[(p_k-2)r+
b_k+r\big]=\sum_{k=1}^{m}c_k(b_k+r)^{\frac{p_k-4}{2}}\big[(p_k-1)r+
b_k\big].
\end{split}\end{equation*}
Therefore
\begin{equation}\label{AYHA}
\begin{split}
\beta\Phi'(r)-\Lambda(r)\,&=\beta\sum_{k=1}^m
c_k(b_k+r)^{\frac{p_k-2}{2}}-\sum_{k=1}^{m}c_k(b_k+r)^{\frac{p_k-4}{2}}\big[(p_k-1)r+
b_k\big]\\
&=\sum_{k=1}^{m}c_k(b_k+r)^{\frac{p_k-4}{2}}\big[\beta(b_k+r)
-(p_k-1)r-
b_k\big]\\
&=\sum_{k=1}^{m}c_k(b_k+r)^{\frac{p_k-4}{2}}\big[(\beta-1)b_k
+(\beta-p_k+1)r\big].
\end{split}
\end{equation}
Also, by~\eqref{Su}, we write~$g=g(\nabla u,u)$, referring to~$\zeta\in\R^n$ as the variable corresponding to~$\nabla u$
and to~$\eta\in\R$ as the variable corresponding to~$u$. By~\eqref{LAMONO},
we know that
\begin{equation}\label{INA1}
g_\eta\ge0.
\end{equation}
Moreover, by the homogeneity of~$g$ in~\eqref{LAMONO2}, we have that
\begin{equation}\label{INA2}
\nabla_\zeta g(\zeta,\eta)\cdot\zeta=\beta g(\zeta,\eta).
\end{equation}
As a consequence, by~\eqref{a1.13}, \eqref{INA1}
and~\eqref{INA2} (and using short notations whenever
possible), we have
\begin{eqnarray*}
\mathscr{R}&=&-\frac{2fg|\nabla u|^2}{\Phi'}+2|\nabla u|^4 g_{\eta}
+\frac{2f |\nabla u|^2}{\Lambda}\nabla_\zeta g\cdot\nabla u\\
&\ge&-\frac{2fg|\nabla u|^2}{\Phi'}+
\frac{2\beta fg |\nabla u|^2}{\Lambda}\\
&=&\frac{2fg|\nabla u|^2}{\Lambda\;\Phi'}\left(
\beta \Phi'-\Lambda\right).
\end{eqnarray*}
Hence, recalling~\eqref{AYHA}, we obtain
\begin{equation}\label{67780}
\mathscr{R}\ge
\frac{2fg|\nabla u|^2}{\Lambda\;\Phi'}
\sum_{k=1}^{m}c_k(b_k+|\nabla u|^2)^{\frac{p_k-4}{2}}\big[(\beta-1)b_k
+(\beta-p_k+1)|\nabla u|^2\big].\end{equation}
Now we claim that
\begin{equation}\label{6778}
\Xi:=fg\sum_{k=1}^{m}c_k(b_k+|\nabla u|^2)^{\frac{p_k-4}{2}}\big[(\beta-1)b_k
+(\beta-p_k+1)|\nabla u|^2\big]\ge0.
\end{equation}
To prove~\eqref{6778} we distinguish six cases, according
to the different assumptions in~\eqref{YAU1}-\eqref{YAUF}. To start with, let us assume that~\eqref{YAU1}
is satisfied. Then, we have that
$$ \Xi=
fg c_1(b_1+|\nabla u|^2)^{\frac{p_1-4}{2}}\big[(\beta-1)b_1
+(\beta-p_1+1)|\nabla u|^2\big]
=fg c_1(b_1+|\nabla u|^2)^{\frac{p_1-4}{2}}(p_1-2)b_1\ge0,$$
and this proves~\eqref{6778} in this case. The same way can be used to discuss cases~\eqref{YAU4} and~\eqref{YAU5}, we omit them here.

If instead~\eqref{YAU2} is satisfied, we find that~$(\beta-1)b_k
+(\beta-p_k+1)|\nabla u|^2\ge0$ and consequently~$\Xi\ge0$,
which shows~\eqref{6778} in this case.

In addition, if~\eqref{YAU3} is satisfied, we see that
$$ \beta-p_k+1\le \beta-p_1+1\le0,$$
and thus
$$ \Xi=fg\sum_{k=1}^{m}c_k |\nabla u|^{{p_k-4}}
(\beta-p_k+1)|\nabla u|^2 \ge0.$$

Finally, if~\eqref{YAUF} holds true, we see that
$$ \Xi=fg c_1(b_1+|\nabla u|^2)^{\frac{p_1-4}{2}}\big[(\beta-1)b_1
+(\beta-p_1+1)|\nabla u|^2\big]
=0
.$$
This completes the proof of~\eqref{6778}.

Then, the desired result follows from~\eqref{67780} and~\eqref{6778}.
\end{proof}

With the previous work, we can now establish Corollary~\ref{CASE2},
which gives a series of concrete situations in which our main gradient
estimate holds true.

\begin{proof}[Proof of Corollary~\ref{CASE2}]
By either~\eqref{EITH1X} or~\eqref{EITH2X}
we have the validity of either~\eqref{EITH1} or~\eqref{EITH2}
and consequently, by Proposition~\ref{STRUCT},
we deduce that either Assumption~A or Assumption~B is satisfied.

This is one of the cornerstones to apply Theorem~\ref{T1.2}.
The other fundamental ingredient to apply Theorem~\ref{T1.2}
lies in the reminder estimate~\eqref{COND}, which we are now going to check.
To this end, we want to exploit Proposition~\ref{CASE}
and, for this, we need to verify that its assumptions are fulfilled in
our setting. Indeed, we have that~\eqref{LAMONO} and~\eqref{LAMONO2}
follow from~\eqref{LAMONOX} and~\eqref{LAMONO2X}. Furthermore,
at least one among~\eqref{YAU1}-\eqref{YAUF}
is satisfied, in light of~\eqref{YAU1X}-\eqref{YAUFX}.
Condition~\ref{Su} is also fulfilled, due to the structure of~$g$ in~\eqref{DIVS}.
Therefore, all the hypotheses of Proposition~\ref{CASE} are satisfied,
and consequently we deduce from Proposition~\ref{CASE} that~$\mathscr{R}\ge0$.

This in turn gives that condition~\eqref{COND} is satisfied and,
as a consequence, we are in the position of exploiting Theorem~\ref{T1.2}.
In this way, the desired result in~\eqref{GRADX}
plainly follows from~\eqref{GRAD}.
\end{proof}
\vskip4mm
\par\noindent
\section{Proof of Theorem~\ref{T1.3}}\label{74d96543}
\setcounter{equation}{0}
\vskip2mm
\noindent
In this section, we prove Theorem~\ref{T1.3}.
After our preliminary work,
this part follows closely some arguments in~\cites{MR1296785,MR3049726}.
We provide full details in the specific
case in which we are interested, for the facility of the reader.

\begin{proof}[Proof of Theorem~\ref{T1.3}]
We take $x_0$ and~$r_0$
as in the statement of Theorem~\ref{T1.3}
and we define
\begin{equation*}
\mathcal{V}:=\{x\in\mathbb{R}^n~\mathrm{s.t}.~u(x)=r_0\}.
\end{equation*}
Notice that~$x_0\in\mathcal{V}$, and hence $\mathcal{V}\neq\varnothing$.
Furthermore, by the continuity of~$u$, we have that~$\mathcal{V}$ is closed. We claim that
\begin{equation}\label{a3.9}
\mathcal{V}\mathrm{~is~also~open}.
\end{equation}
{F}rom this, we would obtain that $\mathcal{V}=\mathbb{R}^n$, which is the thesis of Theorem~\ref{T1.3}. Therefore we focus on the proof of~\eqref{a3.9}.
For this, we fix $\hat{y}\in\mathcal{V}$ and $\hat{w}\in S^{n-1}$.
For any $t\in\mathbb{R}$, we define
\begin{equation*}
\varphi(t):=u(\hat{y}+t\hat{w})-u(x_0).
\end{equation*}
We claim that there exist positive constants $c$ and $C$ such that
\begin{equation}\label{a3.10}
|\varphi'(t)|\leq C|\varphi(t)|,\qquad~\mathrm{for}~\mathrm{all}~t\in(-c,c).
\end{equation}
For this, we define
\begin{equation}\label{hatp}
\hat{p}:=\begin{cases}
p& {\mbox{ if Assumption~A holds with $p>2$,}}\\
2& {\mbox{ otherwise.}}\end{cases}
\end{equation}
We also make use of the function~$\Psi$
introduced in the proof of Lemma~\ref{L1.1},
which satisfies the functional identity
\begin{equation*}
\Psi^{-1}(r)=2r\Phi'(r)-\Phi(r),\qquad{\mbox{for all }}r\in[0,+\infty).
\end{equation*}
Let also
\begin{equation*}
G(r):=\Psi^{-1}(r)-\epsilon r^{\hat{p}/2}.
\end{equation*}
The parameter $\epsilon>0$ will be chosen conveniently small with respect to $M:=\|u\|_{W^{1,\infty}(\mathbb{R}^n)}$ and to
the structural constants given in either~\eqref{a1.7} or~\eqref{a1.9}. Observe that if $M=0$, then $u=0$ in $\mathbb{R}^n$ and so the result is true.

\par
Now we take $r\in(0, M^2]$, with $M>0$,  and $\sigma:=(\sqrt{r},0,\ldots,0)\in\mathbb{R}^n$ and we use~\eqref{a2.1} and~\eqref{a1.5}, and either~\eqref{a1.7} or~\eqref{a1.9},
to see that
\begin{eqnarray}\nonumber
\Lambda(r)&=&2r\Phi''(r)+\Phi'(r)\\\nonumber
&=&|\sigma|^{-2}\sum_{i,j}a_{ij}(\sigma)\sigma_{i}\sigma_{j}\\\nonumber
&\geq&\left\{
   \begin{array}{ll}
C_{1}(a+|\sigma|)^{p-2}\qquad\mathrm{if~Assumption~A~holds~and}~p>2,\\
\displaystyle\frac{C_1}{(a+|\sigma|)^{2-p}}\qquad\mathrm{if~Assumption~A~holds~and}~p\in(1,2],\\\nonumber
\displaystyle\frac{C_1}{1+|\sigma|}\qquad\mathrm{if~Assumption~B~holds}
\end{array}
    \right.\\
&\geq&\left\{
   \begin{array}{ll}
C_{1}|\sigma|^{p-2}\qquad\mathrm{if~Assumption~A~holds~and}~p>2,\\\label{a3.11}
\displaystyle\frac{C_1}{(a+M)^{2-p}}\qquad\mathrm{if~Assumption~A~holds~and}~p\in(1,2],\\
\displaystyle\frac{C_1}{1+M}\qquad\mathrm{if~Assumption~B~holds}
\end{array}
    \right.\\
&\geq&\left\{
   \begin{array}{ll}
\displaystyle\frac{\epsilon p}{2}r^{p/2-1}\qquad\mathrm{if~Assumption~A~holds~and}~p>2,\\
\epsilon\qquad\mathrm{if~Assumption~A~holds~and}~p\in(1,2],\\\nonumber
\epsilon\qquad\mathrm{if~Assumption~B~holds}
\end{array}
    \right.\\
&=&\displaystyle\frac{\epsilon \hat{p}}{2}r^{\hat{p}/2-1},\nonumber
\end{eqnarray}
as long as $\epsilon$ is small enough.

\par
Furthermore, notice that, by $\Phi(0)=0$ and either~\eqref{083-384-1} or~\eqref{083-384-2}, we have that $G(0)=0$. Also, by~\eqref{a2.1}, we have
\begin{equation*}
G'(r):=\Lambda(r)-\frac{\epsilon\hat{p}}{2} r^{\hat{p}/2-1}
\end{equation*}
for any $r>0$ and therefore $G'(r)\geq0$ for any $r\in(0,M^2]$, thanks to~\eqref{a3.11} (as long as $\epsilon$ is small enough). As a consequence, $G(r)\geq0$ and therefore
\begin{equation}\label{a3.120}
\Psi^{-1}(r)\geq \epsilon r^{\hat{p}/2}
\end{equation}
for any $r\in(0,M^2]$. By taking $r:=|\nabla u(\hat{y}+t\hat{w})|^2$ in~\eqref{a3.120} and using~\eqref{GRAD},
we obtain
\begin{equation}\label{a3.12}
\aligned
|\varphi'(t)|^{\hat{p}}&\leq|\nabla u(\hat{y}+t\hat{w})|^{\hat{p}}\\
&\leq\frac{1}{\epsilon}\Psi^{-1}(|\nabla u(\hat{y}+t\hat{w})|^2)\\
&\leq\frac{2}{\epsilon}F(u(\hat{y}+t\hat{w}))\\
&=\frac{2}{\epsilon}\big[F(u(\hat{y}+t\hat{w}))-F(u(x_0))\big].
\endaligned
\end{equation}
Now, we claim that if $r$ is sufficiently close to $r_0$ then there exists $C_0>0$ such that
\begin{equation}\label{a3.13}
\big|F(r)-F(r_0)\big|\leq C_0|r-r_0|^{\hat{p}}.
\end{equation}
To check this we distinguish two cases,
according to the value of~$\hat{p}$.
First of all, if~$\hat{p}=2$,
we use a second order Taylor
expansion of $F$, and we conclude that
$$ \big|F(r)-F(r_0)\big|
\leq \sup_{\rho\in[r_0-1,r_0+1]}|F''(\rho)|\,
|r-r_0|^{2},$$
from which~\eqref{a3.13} plainly follows
in this case.

If, on the other hand, $\hat{p}\ne2$,
then the setting in~\eqref{hatp} gives
that Assumption~A holds true with~$p=\hat{p}>2$.
Then, in this case~\eqref{a3.13}
follows from~\eqref{Z1.16}.
The proof of~\eqref{a3.13} is therefore complete.

Now, plugging~\eqref{a3.13} into~\eqref{a3.12}, we get that there exists $c>0$ small enough such that
\begin{equation*}
|\varphi'(t)|^{\hat{p}}\leq\frac{2C_0}{\epsilon}|u(\hat{y}+t\hat{w})-u(x_0)|^{\hat{p}}=\frac{2C_0}{\epsilon}|\varphi(t)|^{\hat{p}},\quad t\in(-c,c).
\end{equation*}
Taking $C=\left(\frac{2C_0}{\epsilon}\right)^{1/\hat{p}}$, we obtain~\eqref{a3.10}, as desired.

{F}rom~\eqref{a3.10} we obtain that the function $t\mapsto|\varphi(t)|^{2}e^{-2Ct}$ is non-increasing for small $t$. Accordingly, $|\varphi(t)|\leq|\varphi(0)|e^{Ct}=0$ for small $t$, that is $\varphi(t)$ vanishes identically (for small $t$, independently of $\hat{w}$). By varying $\hat{w}$, we
obtain that $u$ is constant in a small neighborhood of $\hat{y}$. This proves~\eqref{a3.9} and thus Theorem~\ref{T1.3}.
\end{proof}
\vskip4mm
\par\noindent

\section*{Acknowledgments}
\vskip2mm
\noindent
Cecilia Cavaterra has been partially supported by GNAMPA (Gruppo Nazionale per
l'Analisi Matematica, la Probabilit� e le loro Applicazioni) of INdAM
(Istituto Nazionale di Alta Matematica).

Serena Dipierro has been supported by the DECRA Project DE180100957
``PDEs, free boundaries and applications''.

Serena Dipierro and Enrico Valdinoci have been supported
by the Australian Research Council Discovery Project
DP170104880
``N.E.W.
Nonlocal Equations at Work''.

Zu Gao has been supported by the Chinese Scholarship Council.
This work was written on the occasion of a very pleasant and fruitful visit
of Zu Gao at the Universit\`a di Milano, which we thank for the warm
hospitality.

\end{document}